\documentclass[reqno,10pt]{amsart}

\usepackage{amsmath,amsfonts,amsthm}
\usepackage{enumerate,enumitem,color,bm}
\usepackage{tikz}
\usepackage[normalem]{ulem}
\usepackage[mathscr]{eucal}
\usepackage{mathabx} % for \widecheck

\usepackage{marvosym} % for biological symbols
\usepackage{amstext}
\usepackage{etoolbox}
\pretocmd\mvchr{\text}{}{\errmessage{Patching \noexpand\mvchr failed}}
\pretocmd\textmvs{\text}{}{\errmessage{Patching \noexpand\textmvs failed}}

\newtheorem{theorem}{Theorem}
\newtheorem{corollary}{Corollary}

\newtheorem{proposition}{Proposition}
\newtheorem{remark}{Remark}

\def\bds{\begin{displaystyle}}
\def\eds{\end{displaystyle}}
\renewcommand{\P}{{\mathbb{P}}}
\newcommand{\E}{{\mathbb{E}}}
\newcommand{\D}{{\mathbb{D}}}
\newcommand{\R}{{\mathbb{R}}}
\newcommand{\T}{{\mathbb{T}}}
\newcommand{\A}{{\mathbb{A}}}
\newcommand{\K}{{\mathbb{K}}}
\renewcommand{\S}{{\mathbb{S}}}
\newcommand{\N}{{\mathbb{N}}}
\newcommand{\1}{{\mathbf{1}}}
\renewcommand{\i}{{\bm{i}}}
\newcommand{\q}{{\quad}}
\newcommand{\F}{{\Female}} %symbol
\newcommand{\M}{{\Male}} %symbol
\newcommand{\FM}{{\FemaleMale}} %symbol
\newcommand*{\Cdot}{\raisebox{-0.5ex}{\scalebox{2}{$\cdot$}}}

\newcommand{\spot}{\marginpar{$\bullet$\kern30pt}} % Used to mark changes

\begin{document}

\title[General, Size Dependent Populations]{Limit theorems for multi-type general branching processes with population dependence}

\author[Fan et al.]{Jie Yen Fan}\address{School of Mathematical Sciences, Monash University, Clayton, Vic. 3058, Australia.}\email{jieyen.fan@monash.edu}
\author[]{Kais Hamza}\address{School of Mathematical Sciences, Monash University, Clayton, Vic. 3058, Australia.}\email{kais.hamza@monash.edu}
\author[]{Peter Jagers}\address{Department of Mathematical Sciences, Chalmers University of Technology.}\email{jagers@chalmers.se}
\author[]{Fima C. Klebaner}\address{School of Mathematical Sciences, Monash University, Clayton, Vic. 3058, Australia.}\email{fima.klebaner@monash.edu}

\thanks{Supported by the Australian Research Council Grants 
DP120102728, DP150103588, and the Knut and Alice Wallenberg Foundation}



\begin{abstract} 
A general multi-type population model is considered, where individuals live and reproduce according to their age and type, but also under the influence of the size and composition of the entire population. We describe the dynamics of the population density as a measure-valued process and obtain its asymptotics, as the population grows with the environmental carrying capacity. ``Density'' in this paper generally refers to the population size as compared to the carrying capacity. Thus, a deterministic approximation is given, in the form of a Law of Large Numbers, as well as a Central Limit Theorem. Migration can also be incorporated. 
This general framework is then adapted to model sexual reproduction, with a special section on serial monogamic mating systems.
\end{abstract}

\maketitle

%%%%%%%%%%%%%%%%%%%%%%%%%%%%

\section{Introduction}

Classical stochastic population dynamics (branching processes of various generality) assume independently acting individuals, usually in stable circumstances, whereas deterministic approaches, while paying attention to the feedback loop between a population and its environment,  tend to sweep dependence among individuals under the carpet. In a series of papers, branching processes with population size dependence have been studied, first in the discrete time Galton-Watson case \cite{Kleb84,Kleb89}, more recently also for single-type general processes \cite{FanEtal19,HJK,HamJagKle13,JagKle00,JagKle11}, where \cite{FanEtal19} can be viewed as a single-type companion paper to the present. A first approach to general multi-type processes, comprising sexual reproduction, was made in \cite{JagKle16}. 

In this paper, we introduce a general model, where an individual is characterised by its age and type. Reproduction and death may depend on these as well as on the environment (the size and composition of the whole population). 
This is achieved by describing the population as a measure-valued process, as done in \cite{FanEtal19}. 
The dependence on the composition of the population could be, for instance, on the size of subpopulations of different types or ages. 
The resulting general model has a high level of flexibility and can be
adapted to the dynamics of quite diverse biological populations.   

In an ecological community, species interact; reproduction and death depend on the composition of the community. This complex structure can be approached in our setup. 
In particular, in predator-prey models, species can be viewed as types, and in food webs (e.g. \cite{Rossberg}), many complex interactions occur, where age, sex, and other types are important. 

Another example is the modelling of cell proliferation and differentiation, which has been studied by multi-type branching processes (e.g. \cite{HyrEtal05}, \cite{NorEtal11}). One particular phenomenon is that of the formation of oligodendrocytes: 
progenitor cells dividing into new progenitor cells, differentiating into oligodendrocytes, or dying. Here a detailed description of the cell population and its evolution can be given, for high cell densities.

In the second part (Section \ref{S:Monogamy}) of this paper, we give a special case of sexual reproduction, where couple formation is considered, in a (serial) monogamy system. In nature such systems can have quite varying forms from season-long to life-long couples, persisting  until the death of one of the partners. The approach can be adapted to different strategies of couple formation.  

Proofs are to be found in Appendices --  
Appendix \ref{S:Appendix-Multitype} for Section \ref{S:Multitype} and Appendix \ref{S:Appendix-Monogamy} for Section \ref{S:Monogamy}. 

Naively speaking, our object of study is the process describing how
the number of individuals $S^K_t(B\times A)$, with ages in an interval $A$
and types in a set  $B$, evolves as time $t\geq 0$ passes, for large $K>0$ . Here, $K$
historically is the habitat carrying capacity, more generally
interpretable as a system size parameter. Mathematically, thus, we
consider measure-valued processes $(S^K_t)_{t\geq 0}$, as $K\to\infty$. 
The process is supposed to start from a measure $S^K_0$, 
the total mass of which is of the order $K$, 
and the evolvement of  $S^K_t$ is governed by birth and death intensities 
that depend on individual age and type, population structure, and $K$. 

\section{Multi-type population structure dynamics} \label{S:Multitype}

\subsection{The model set-up}

The population model is defined as in \cite{FanEtal19}, but with some additional characteristics.  
It builds upon the Ulam-Harris family space, with type inherited from mother to child, 
as described in \cite{JagersMarkov}. 
An individual $x = x_1x_2\cdots x_n$ is thought of as 
the $x_n$th child of $\cdots$ of the $x_2$th child of the $x_1$th ancestor and 
$$ I := \bigcup_{n=1}^\infty \N^n $$ denotes the set of possible individuals. 
In particular, $xj$ denotes the $j$th child of individual $x$. 

Each individual $x$ born into the population is characterised by
its type, $\kappa_x$, and birth time, $\tau_x$, which is recursively
defined as the birth time of its mother plus her age at
bearing the individual, cf. \cite{Harris,Jagers,JagersMarkov}. 
Individuals age at rate 1 until death. 
Let $\lambda_x$ denote the lifespan of $x$ and 
write $\sigma_x = \tau_x+\lambda_x$ for the death time. 
Then at each time $\tau_x \leq t< \tau_x+\lambda_x$ 
the individual will be in {\em state} $s_x (t)= (\kappa_x,t-\tau_x) \in \K\times\A =: \S$. 
We write $\K$ for the set of possible types, assumed to be finite; 
and $\A$ for the set of possible ages, assumed to be a bounded interval $[0,\omega]$ with $\omega<\infty$ denoting the maximal age (following classical demographic notation), which in the present case will be defined in Section \ref{S:LimitThms}. 

The composition of the population at time $t$ can be represented by the measure
\begin{equation} \label{E:S}
S_t(di,dv) = \sum_{x\in I} \mathbf{1}_{\tau_x\le t<\sigma_x} \delta_{(\kappa_x,t-\tau_x)}(di,dv),
\end{equation}
where $\delta_{s}$ denotes the Dirac measure at $s$, assigning unit mass to $s$. 
Thus, $S_t$ is a measure with unit mass at the state of each
individual that is alive at time $t$.
(In this and later we allow ourselves to suppress writing the
dependence upon carrying capacity $K$ and often also upon time $t$.) 
We shall denote the set of finite non-negative measures on $\S$, 
with its weak topology, by $\mathcal{M}(\S)$, or $\mathcal{M}$ for short. 
Thus $S_t \in \mathcal{M}$ for each $t$. Here $\S$ has the product
topology of discrete and Euclidean topology, of course. 

The initial population $S_0$ is assumed to be finite and deterministic. 
Suppose that bearing and death times are stochastically given by type,
age  and population dependent rates. 
An individual of state $s$ in a population composition $S$
gives birth at rate $b_S(s)$ until it dies, with the death intensity being $h_S(s)$.
At each birth event, a number of offspring of type $i\in \K$ is generated, 
with the distribution of a random variable $\widecheck \xi^i_S(s)$. 
Similarly, offspring may be born at the death of the mother (splitting) 
with the distribution of the random variable $\widehat \xi^i_S(s)$. 
We reiterate that the suffix $S$ represents the population
composition, which could include population size and other aspects of
the population structure. 
Alternatively, the reproduction process could have been given as an
integer valued random measure on $\S$, like in \cite{JagersMarkov},
disintegrated into a stream of events and a random mass at each
event.  The variables  $\widecheck \xi^i_S(s)$ then have the Palm
distribution, given a birth event at $s\in \S$, \cite{PalmJagers,Kallenberg}.  
Those pertaining to the same mother but at different $s$ are also
assumed independent.  

\begin{remark} 
Indeed, the birth rate pertaining to each individual $x$ with $\tau_x<\infty$ 
gives rise to a point process of bearings by $x$ at 
ages $0< \alpha_{x}^{1}< \alpha_{x}^{2}< \ldots <\lambda_x$.  
(For simplicity, we disregard the possibility of immediate bearing at birth, \cite{Jagers}. 
Births at death (splitting) are handled separately in this paper.) 
The random variable $\widecheck \xi^{x,i}_{S_{\tau_x+\alpha_{x}^{j}}}(\kappa_x,\alpha_{x}^{j})$ 
that stands for the number of type $i$ children born by $x$ at time $\tau_x+\alpha_{x}^{j}$, 
then follows the distribution of $\widecheck \xi^i_S(s)$, with $S=S_{\tau_x+\alpha_{x}^{j}}$ and $s=(\kappa_x,\alpha_{x}^{j})$. 
A corresponding remark is valid for the number of children generated at death. 
We shall not enter into the awkward details of this but refer to the construction in \cite{JagersMarkov}. 
At the individual level, the construction there is, however, more general than the present 
not only through a richer type space but also since earlier reproduction history may influence the propensity to give birth. 
\end{remark}

The population model can be described in either of two ways, through the generator of
the process (as in \cite{JagKle00} and \cite{JagKle16}),  
or through the evolution equation of the process (as in
\cite{FanEtal19}), see Section \ref{SetupProof}.
Before giving the dynamic equation, we clarify the concepts of differentiation and integration on $\S$. 
Derivatives of a function on $\S $ refer to the derivatives with
respect to the second, continuous variable, i.e. age.
In particular, for $f:\S\to\R$ and $s=(i,v)$, we write $f^{(j)}$ to mean 
$$f^{(j)}(s) = f^{(j)}(i,v) = \partial_v^{j}f(i,v),$$
where $\partial_v^{j}$ denotes the $j$th derivative with respect to
the variable $v$. 
We also use $f'$ for $f^{(1)}$.
If $\mu$ is a Borel (positive or signed) measure on $\S$, then for $f:\S\to\R$, we write 
$$(f,\mu) = \int_\S f(s) \mu(ds) = \int_{\K\times\A} f(i,v)
\mu(di\times dv) = \sum_{i\in\K} \int_{\A} f(i,v) \mu(\{i\}\times dv).$$
For non-negative integers $j$, we write $C^j(\S)$ for the space of functions on $\S$ with continuous derivatives (with respect to the age variable) up to order $j$.
Since we only consider a bounded domain $\S$, functions in $C^j(\S)$
are bounded and so are the $j$ derivatives.
We can define the norm
$$||f||_{C^j(\S)} = \max_{0\le \iota \le j} \sup_{s\in \S} |f^{(\iota)}(s)| = \max_{0\le \iota \le j} \sup_{i\in\K,v\in\A} |f^{(\iota)}(i,v)|,$$
and will use $||\cdot||_{C^0}$ and $||\cdot||_\infty$ interchangeably. 

We write $\widecheck m^i_S(s) = \E[\widecheck \xi^i_S(s)|S]$, $i\in\K$, 
for the expectation of the number of type $i$ progeny at the birth event in question, 
given population size and composition, and the individual's state at that time. 
Note that the expectation is that of a random variable having the
specified conditional distribution and that the conditional covariances of the
number of children born by the same mother at two different bearing
events vanish. 
Similarly, we write $\widecheck\gamma^{i_1i_2}_S(s) = \E[\widecheck \xi^{i_1}_S(s)\widecheck \xi^{i_2}_S(s)|S]$, $i_1,i_2 \in\K$ 
and  define $\widehat m^i_S(s)$ and $\widehat\gamma^{i_1i_2}_S(s)$ for $\widehat \xi$ in a similar vein.
Then, the mean intensity of births of an individual of state $s$ at time $t$ is 
$\sum_i \widecheck m^i_{S_t}(s)b_{S_t}(s) + \sum_i \widehat
m^i_{S_t}(s)h_{S_t}(s)$. 

For $f\in C^1(\S)$, 
\begin{equation} \label{E:fSt}
(f,S_t) = (f,S_0) + \int_0^t (L_{S_u}f,S_u) du + M^f_t,
\end{equation}
where
$$L_Sf = f' - h_Sf + \sum_{i\in\K} f(i,0) \Big( b_S \widecheck m_S^i + h_S \widehat m_S^i \Big)$$
and $M^f_t$ is a locally square integrable martingale with predictable quadratic variation
\begin{multline*}
\big<M^f\big>_t = \int_0^t \bigg(
\sum_{i_1\in\K} \sum_{i_2\in\K} f(i_1,0)f(i_2,0)
\Big( b_{S_u} \widecheck\gamma^{i_1i_2}_{S_u} + h_{S_u} \widehat\gamma^{i_1i_2}_{S_u} \Big) \\
+ h_{S_u} f^2 - 2 \sum_{i\in\K} f(i,0) h_{S_u} \widehat m^i_{S_u} f , S_u \bigg) du.
\end{multline*}
For simplicity of notation, we shall write
$$n^i = b \widecheck m^i + h \widehat m^i \quad \textnormal{and} \quad
w^{i_1i_2} = b \widecheck\gamma^{i_1i_2} +h  \widehat\gamma^{i_1i_2}$$
from here onwards. 

Differential equation for specific characteristic of the population can be obtained from an appropriate test function $f$ (and $F$). 
For example, if $f=1$, the result is the population size; taking $f(i,v)=v$, yields the sum of the ages of the population. It is also possible to count individuals of a certain type; for instance, taking $f(i,v) = \mathbf{1}_{i=1}$, we have the size of the subpopulation of type 1, and similarly the average age of individuals of a type can be obtained.

\begin{remark}
The above results can be extended to test functions on $\S\times\T$, where $\T$ is a time interval $[0,T]$. 
Consider test functions $f(s,t) \equiv f(i,v,t)$ on $\S\times\T$. We shall write $f_t(s)$ to mean $f(s,t)$ and use the two notations interchangeably. 
For such $f$, $\partial_1f$ refers to the derivative with respect to the variable $s$ (or $v$ in this case), and $\partial_1f$ refers to the derivative with respect to $t$. 
For $f\in C^{1,1}(\S\times\T)$,
$$\big(f_t,S_t\big) = \big(f_0,S_0\big)
+ \int_0^t \Big( \partial_1f_u + \partial_2f_u 
- f_uh_{S_u} + \sum_{i\in\K} f_u(i,0)n^i_{S_u}, S_u\Big) du
+ M^f_t,$$
where $M^f_t$ is a martingale with the predictable quadratic variation
\begin{multline} \label{E:QVMf2v}
\big<M^f\big>_t
= \int_0^t \bigg(
\sum_{i_1\in\K} \sum_{i_2\in\K} f_u(i_1,0)f_u(i_2,0) w^{i_1i_2}_{S_u} + h_{S_u} f^2_u \\
- 2 \sum_{i\in\K} f_u(i,0) h_{S_u} \widehat m^{i}_{S_u} f_u, S_u \bigg) du.
\end{multline}
\end{remark}

\begin{remark} \label{R:MeasureM}
Using the same argument as in \cite[Section 5.2]{FanEtal19},
we can show the existence of a measure $M_t$ such that $(f,M_t) = M^f_t$,
and define the martingale $\int_0^t (f_u,dM_u)$ for $f\in C^0(\S\times\T)$.
(See Section \ref{S:MeasureMProof}.)
\end{remark}

\subsubsection{Applications: sexual reproduction} \label{S:ModelApplic1}
A simple application of the stochastic process introduced above, where individual life can be influenced by many factors like individual type and age, population size, and population structure, 
is to model sexual reproduction. 
Let $\K = \{1,2\} \equiv \{\F,\M\}$ with type 1 (denoted as \F) representing females (the reproducing type)  and type 2 (denoted as \M), representing males,
so that $b_S(\F,v)\ge0$ and $b_S(\M,v)=0$ for any $S\in\mathcal{M}(\S)$ and $v\in\A$.
Also, for $\xi = \widecheck \xi, \widehat \xi$ and $i=\F,\M$, 
$\xi^i_S(\F,v)\ge0$ and $\xi^i_S(\M,v)=0$; and
for $m = \widecheck m, \widehat m$, $\gamma = \widecheck \gamma, \widehat \gamma$, and $i,i_1,i_2=\F,\M$, 
$m^i_S(\M,v) = \gamma^{i_1,i_2}_S(\M,v) = 0$.
Writing $S^\F(dv) = S(\{\F\},dv)$ for the age structure of the female subpopulation and $S^\M(dv) = S(\{\M\},dv)$ for that of the male subpopulation, we can obtain equations governing the subpopulations. 
Note that $(f,S) = (f(\F,\cdot),S^\F) + (f(\M,\cdot),S^\M)$. 
For simplicity of notation, we shall write
$\mathfrak b_S(v)=b_S(\F,v)$, $\widehat{\mathfrak m}^i_S(v)= \widehat m^i_S(\F,v)$, $\mathfrak n^i_S(v)=n^i_S(\F,v)$ and $\mathfrak w^{i_1i_2}_S(v)=w^{i_1i_2}_S(\F,v)$.
Taking $f(i,v) = f_\F(v)\mathbf{1}_{i=\F}$ in \eqref{E:fSt} gives the age structure of the female subpopulation: 
$$(f_\F,S^\F_t) = (f_\F,S^\F_0) + \int_0^t \big(f'_\F-h_{S_u}(\F,\cdot)f_\F +f_\F(0)\mathfrak n^\F_{S_u}, S^\F_u\big) du + M^\F_t $$
with
$$\big<M^\F\big>_t = \int_0^t \big( f_\F(0)^2 \mathfrak w^{\F\F}_{S_u} + h_{S_u}(\F,\cdot) f_\F^2 - 2 f_\F(0) h_{S_u}(\F,\cdot) \widehat{\mathfrak m}^\F_{S_u} f_\F, S^\F_u \big) du; $$
whereas taking $f(i,v) = f_\M(v)\mathbf{1}_{i=\M}$ in \eqref{E:fSt} gives the dynamics of the male subpopulation: 
$$(f_\M,S^\M_t) = (f_\M,S^\M_0) + \int_0^t \big(f'_\M-h_{S_u}(\M,\cdot)f_\M, S^\M_u\big) du 
+f_\M(0) \int_0^t \big(\mathfrak n^\M_{S_u},S^\F_u\big) du + M^\M_t$$
with
$$\big<M^\M\big>_t = f_\M(0)^2 \int_0^t \big( \mathfrak w^{\M\M}_{S_u}, S^\F_u \big) du + \int_0^t \big( h_{S_u}(\M,\cdot) f_\M^2, S^\M_u \big) du. $$

This approach is different from the so-called bisexual branching process,   
where individuals mate to form couples, which then reproduce like in a
Galton-Watson type process in discrete time.  Those were introduced by
Daley \cite{Daley1968} and have been studied by many others
(cf. \cite{MolMotRam06} and references therein). 
A (pseudo) continuous time bisexual branching process, was given in
\cite{MolYan03}, mating supposed to occur at the events of a point
process and individuals born by couples having independent and identically distributed life spans. 
As opposed to that, %where births are governed by mating units, 
our model is individual based, and females reproduce, with an
intensity, which may depend on the availability of males, among other things. 
Sexual reproduction without couple formation, like in fish, as well as
asexual reproduction, can be easily handled within this 
framework. Couple formation and sexual reproduction with mating, like
in many higher animals, can be captured, to some extent,  by careful
choice of rates and offspring distribution. 

\subsection{The limit theorems} \label{S:LimitThms}

We consider a family of population processes as above, indexed by some parameter $K\ge1$, which as mentioned, arises from the notion of carrying capacity of the habitat and often can be interpreted as some size where a population neither tends to increase nor decrease systematically. However, it needs not play such a role, and can be viewed just as a natural system scaling parameter, since we consider starting populations whose size is of the order $K$. The purpose is to establish the asymptotic behaviour as $K$ increases, under the assumption that the dynamics of the population depends on $K$ through the reproduction parameters.
For notational consistency with limits, we write the parameters in the form  $q^K_{S/K}(s)$, for $q=b,h,m,\gamma$. The conditions stated for $m$ and $\gamma$ are to be satisfied by
all $\widecheck m^i$, $\widehat m^i$, $\widecheck\gamma^{i_1i_2}$ and $\widehat\gamma^{i_1,i_2}$
for any $i,i_1,i_2 \in \K$.

We consider a finite time interval $\T= [0,T]$. Thus, to obtain a bounded age space $\A$, it suffices to assume that the age of the oldest individual over the family of processes at time 0, $a^*$, is finite:
$$a^* := \sup_{K\ge1} \sup_{i\in\K} \big( \inf\{v>0: S_0^K (i,(v,\infty))=0\} \big) <\infty.$$
Then the age of the oldest individual in the population at time $t$
cannot exceed $t+a^*$, and we can take $\A=[0,\omega]$ with
$\omega=T+a^*$ as the age space, and $\S=\K\times\A$ the individual state space.

Specifically, we establish the Law of Large Numbers and the Central Limit Theorem associated with our age and type structure process. 
For a simple case of sexual reproduction (cf. Section \ref{S:ModelApplic1}) the Law of Large Numbers was stated in \cite{JagKle16}, without proof. Here, we give and prove results for more general cases. 
The Law of Large Numbers is stated in Section \ref{S:LLN} with proof in Section \ref{S:LLNProof}, and the Central Limit Theorem is stated in Section \ref{S:CLT} with proof in Section \ref{S:CLTProof}.

\subsubsection{Law of Large Numbers} \label{S:LLN}

Denote by $\bar S^K_t = S^K_t/K$ the \emph{density} of the population.
For any $f \in C^1(\S)$ and $t\in\T$,
\begin{equation} \label{E:SbarK}
(f,\bar S^K_t) = (f,\bar S^K_0) + \int_0^t (L^K_{\bar S^K_u}f, \bar S^K_u) du + \frac1K M^{f,K}_t,
\end{equation}
where
\begin{equation} \label{E:LK}
L^K_Sf = f' - h^K_Sf + \sum_{i\in\K} f(i,0) n_S^{i,K}
\end{equation}
and $M^{f,K}_t$ is a locally square integrable martingale with predictable quadratic variation
\begin{multline*}
\big<M^{f,K}\big>_t = \int_0^t \bigg(
\sum_{i_1\in\K} \sum_{i_2\in\K} f(i_1,0)f(i_2,0) w^{i_1i_2,K}_{\bar S^K_u} + h^K_{\bar S^K_u} f^2 \\
- 2 \sum_{i\in\K} f(i,0) h^K_{\bar S^K_u} \widehat m^{i,K}_{\bar S^K_u} f, S^K_u \bigg) du.
\end{multline*}
We have the Law of Large Numbers when the population process is {\em demographically smooth}, i.e. satisfies the following conditions:
\begin{enumerate}
\item [(C0)] The model parameters $b$, $h$, $m$ and $\gamma$ are uniformly bounded, 
i.e., for $q=b,h,m, \gamma$, $\sup q^K_\mu(s) <\infty$, where supremum
is taken over $s\in\S$, $K\ge1$ and $\mu\in \mathcal{M}(\S)$. 
% where $\mathcal{M}(\S)$ denoting the set of finite positive measures on $\S$. 
\item [(C1)]
The model parameters $b$, $h$ and $m$ are normed uniformly Lipschitz in the following sense:
for $q=b,h,m$, there exists $c>0$ such that for all $K\ge1$, $||q^K_\mu - q^K_\nu||_\infty \le c ||\mu-\nu||$,
with $||\mu|| := \sup_{||f||_\infty\le1, f \textnormal{ continuous}} |(f, \mu)|$.
\item [(C2)]
For $q=b,h,m$, the sequence $q^K$ converges (pointwise in $\mu$ and uniformly in $s$)
to $q^\infty_\mu(s) := \lim_{K\to\infty} q^K_{\mu}(s)$.
\item [(C3)]
$\bar S^K_0$ converges weakly to $\bar S_0$
and $\sup_K (1,\bar S^K_0) <\infty$,
that is, the process stabilises initially.
\end{enumerate}

\begin{theorem} \label{T:LLN}
Under the smooth demography conditions (C0)-(C3), the measure-valued process
$\bar S^K$ converges weakly in the Skorokhod space
$\D(\T,\mathcal{M}(\S))$, as $K\to\infty$, to a deterministic
measure-valued process $\bar S$, satisfying
\begin{equation} \label{E:fSInfty}
(f, \bar S_t) = (f,\bar S_0) + \int_0^t (L^\infty_{\bar S_u} f, \bar S_u) du
\end{equation}
where
\begin{equation} \label{E:Linfty}
L^\infty_Sf = f' - h^\infty_Sf + \sum_{i\in\K} f(i,0) n_S^{i,\infty},
\end{equation}
for any $f\in C^1(\S)$ and $t\in\T$.
\end{theorem}

\subsubsection{Relevant spaces and embeddings}

Before stating the Central Limit Theorem, we introduce some Sobolev spaces on $\S$ and the relevant duals, where the convergence will take place.
The compactness of the space $\S$ allows us to work with the classical Sobolev spaces, instead of the weighted Sobolev spaces like other scholars (e.g. \cite{Bor90,Mel98,Met87}) did. 

For $j\in\N_0$, let $W^j(\S)$ be the closure of $C^\infty(\S)$ with respect to the norm
$$||f||_{W^j(\S)} = \bigg(\sum_{\iota=0}^j \sum_{i\in\K} \int_{\A} \big(f^{(\iota)}(i,v)\big)^2 dv \bigg)^{1/2},$$
where $f^{(\iota)}$ are the (weak) derivatives of $f$. In other words,
$W^j(\S)$ is the set of functions $f$ on $\S$ such that $f$ and its
weak derivatives up to order $j$ have a finite $L^2(\S)$ norm,
$$||f||_{L^2(\S)} = \bigg(\sum_{i\in\K} \int_{\A} f(i,v)^2 dv \bigg)^{1/2}. $$
The space $W^j(\S)$ is a Hilbert space.
As in \cite{FanEtal19}, we have the following embeddings:
$$C^j(\S) \hookrightarrow W^j(\S) \ ,\
W^{j+1}(\S) \hookrightarrow C^j(\S) \q \textnormal{and} \q
W^{j+1}(\S) \underset{H.S.}{\hookrightarrow} W^j(\S),$$
where $H.S.$ stands for Hilbert-Schmidt embedding. 

The dual spaces $C^{-j}(\S)$ and $W^{-j}(\S)$
of $C^{j}(\S)$ and $W^{j}(\S)$ respectively, are embedded as follows: 
$$W^{-j}(\S) \hookrightarrow C^{-j}(\S) \ ,\
C^{-j}(\S) \hookrightarrow W^{-(j+1)}(\S) \ \ \textnormal{and} \ \
W^{-j}(\S) \underset{H.S.}{\hookrightarrow} W^{-(j+1)}(\S),$$
and in particular, 
$$C^{-0}(\S) \hookrightarrow C^{-1}(\S) \hookrightarrow W^{-2}(\S) \hookrightarrow W^{-3}(\S) \underset{H.S.}{\hookrightarrow} W^{-4}(\S).$$
The Hilbert-Schmidt embedding plays an important role, since with
this, a ball in the smaller space ($W^{-3}$) is precompact in the larger space ($W^{-4}$). This fact renders it possible to prove the coordinate tightness of the fluctuation process. The other embeddings are required due to the derivative in the definition of the operator $L^K_S$. 

For ease of notation, we suppress writing $(\S)$ for the spaces from here onwards,
that is, we will write $W^j$ and $C^j$ to mean $W^j(\S)$ and
$C^j(\S)$. Further, $L^\infty$ denotes the set of all bounded
measurable functions on $\S=\K\times \A$ with its natural product
sigma-algebra $\mathcal{B}(\S)$. 

\subsubsection{The Central Limit Theorem} \label{S:CLT}

Let $Z^K = \sqrt{K}(\bar S^K - \bar S)$.
Then, for any $f\in C^1$ and $t\in\T$,
\begin{equation} \label{E:ZKL}
(f,Z_t^K) = (f,Z_0^K)
+ \sqrt{K} \int_0^t \big( L^K_{\bar S^K_u}f - L^\infty_{\bar S_u}f , \bar{S}_u \big) du 
+ \int_0^t \big( L^K_{\bar S^K_u}f , Z^K_u \big) du + \tilde{M}_t^{f,K},
\end{equation}
where $L^K_S$ and $L^\infty_S$ are defined as in \eqref{E:LK} and \eqref{E:Linfty},
and $\tilde{M}_t^{f,K}$ is a square integrable martingale with predictable quadratic variation
\begin{multline*}
\big<\tilde M^{f,K}\big>_t = \int_0^t \bigg(
\sum_{i_1\in\K} \sum_{i_2\in\K} f(i_1,0)f(i_2,0) w^{i_1i_2,K}_{\bar S^K_u} + h^K_{\bar S^K_u} f^2\\
- 2 \sum_{i\in\K} f(i,0) h^K_{\bar S^K_u} \widehat m^{i,K}_{\bar S^K_u} f , \bar S^K_u \bigg) du.
\end{multline*}

In addition to (C0)--(C3), we impose the following assumptions:
\begin{enumerate}
\item[(A0)]
Conditions (C1) and (C2) hold also for $q=\gamma$.
\item[(A1)]
$\Xi := \sup_{i,s,S,K} \widecheck \xi^{i,K}_S(s) \vee \sup_{i,s,S,K}
\widehat \xi^{i,K}_S(s)$ is square integrable.  
\item[(A2)]
The reproduction parameters $b_S^K(s)$, $h^K_S(s)$ and $m^K_S(s)$ and their limits (in the sense of (C2)) are in the space $C^4$ (in the argument $s$) with convergence in $C^4$.
Moreover,
$\sqrt K \sup_\mu ||q^K_\mu-q^\infty_\mu||_\infty \rightarrow 0$ as $K\to\infty$,
$\sup_{K,\mu} ||q^K_\mu||_{C^3} <\infty$ and $\sup_\mu ||q^\infty_\mu||_{C^4} <\infty$,
for $q=b,h,m$.
\item[(A3)]
The limiting parameters (seen as functions of $S$) are Fr\'echet differentiable at every $S$.
Namely, for $q=b,h,m$, for every $\mu$,
there exists a continuous linear operator $\partial_S q^\infty_{\mu} :
W^{-4} \to L^{\infty}$ such that 
$$\lim_{||\nu||_{W^{-4}}\to 0} \frac{1}{||\nu||_{W^{-4}}} ||q^\infty_{\mu+\nu}-q^\infty_{\mu}-\partial_S q^\infty_{\mu}(\nu)||_\infty=0.$$
Moreover, $\sup_{\mu} ||\partial_S q^\infty_{\mu}||_{\mathbb{L}^{-4}} \le c$,
where $\mathbb{L}^{-4} = L(W^{-4},L^\infty)$ denotes the space of
continuous linear mappings from $W^{-4}$ to $L^\infty$. 
\item[(A4)]
$Z^K_0$ converges to $Z_0$ in $W^{-4}$
and $\sup_K ||Z^K_0||_{W^{-2}} <\infty$.
\end{enumerate}
Then the fluctuation process $Z^K$ converges.

\begin{theorem} \label{T:CLT}
Under assumptions (C0)--(C3) and (A0)--(A4),
as $K\to \infty$, the process $(Z^K_t)_{t\in\T}$ converges weakly in $\mathbb{D}(\T,W^{-4})$
to $(Z_t)_{t\in\T}$ satisfying, for $f\in W^4$ and $t\in\T$,
\begin{multline} \label{E:fZInfty}
(f,Z_t) = (f,Z_0)
+ \int_0^t \Big( - \partial_S h^\infty_{\bar{S}_u}(Z_u) f + \sum_{i\in\K} f(i,0) \partial_S n^{i,\infty}_{\bar{S}_u}(Z_u) , \bar{S}_u \Big) du \\
+ \int_0^t \Big( f' - h^\infty_{\bar{S}_u} f + \sum_{i\in\K} f(i,0) n^{i,\infty}_{\bar{S}_u} , Z_u \Big) du + \tilde{M}_t^{f,\infty},
\end{multline}
where $\tilde{M}_t^{f,\infty}$ is a continuous Gaussian martingale with predictable quadratic variation
$$\big<\tilde M^{f,\infty}\big>_t = \int_0^t \bigg(
\sum_{i_1\in\K} \sum_{i_2\in\K} f(i_1,0)f(i_2,0) w^{i_1i_2,\infty}_{\bar S_u}+ h^\infty_{\bar S_u} f^2 
- 2 \sum_{i\in\K} f(i,0) h^\infty_{\bar S_u} \widehat m^{i,\infty}_{\bar S_u} f , \bar S_u \bigg) du.$$
\end{theorem}

For each $t$, the limit $Z_t$ takes value in $W^{-4}$. Under certain condition, its expectation corresponds to a signed measure. This is made precise in the following proposition. 

\begin{proposition} \label{P:EZtMeasure}
Suppose that $\partial_Sh^\infty_{\mu}(B)(s)$ is of the form $(g^h(\mu,s,\cdot),B)$ where $g^h(S,s,\cdot) \in W^4$ with $\sup_{S,s} ||g^h(S,s,\cdot)||_{W^4} <\infty$,
and similarly, $\partial_Sn^{i,\infty}_{\mu}(B)(s)$ is of the form $(g^{n,i}(\mu,s,\cdot),B)$ where $g^{n,i}(S,s,\cdot) \in W^4$ with $\sup_{S,s} ||g^{n,i}(S,s,\cdot)||_{W^4} <\infty$.
Then for each $t$, $\nu_t: W^4 \ni f\mapsto \E[(f,Z_t)]$ defines a signed measure on
$(\S, \mathcal{B}(\S))$. 
\end{proposition}

We give the proofs in Appendices; that of Theorem \ref{T:LLN} in Section \ref{S:LLNProof}, that of Theorem \ref{T:CLT} in Section \ref{S:CLTProof}, and that of Proposition \ref{P:EZtMeasure} in Section \ref{S:LastProof}.

Loosely speaking, the Law of Large Numbers gives the first order approximation to the population when it is large; the Central Limit Theorem gives the second order approximation. 
Quantities of interest, such as the population size or the number of individuals of certain type within certain age interval, can be computed simply by choosing the appropriate test functions. 
In specific cases, where explicit model parameters are available, more information can be obtained and inference be made from observations. 

\subsubsection{Applications: sexual reproduction, revisited} \label{S:ModelApplic2}

We continue from Section \ref{S:ModelApplic1}, 
revising notation slightly, writing parameters in the form $q^K_{\bar S}(s)$. 
Suppose that all the conditions (C0)-(C3) and (A0)-(A4) are satisfied, 
so that the Law of Large Numbers and Central Limit Theorem hold. 
For example, $q$ may take the form $q\big(i,v,(1,\bar S),(g,\bar S)\big)$ with $g(i,v) = \mathbf{1}_{i=\F}$. 
Then, the limiting parameters have the same form and the Fr\'echet derivatives are
$$\partial_S q^\infty_\mu(B)(i,v) = \partial_3 q\big(i,v,(1,\mu),(g,\mu)\big)(1,B) + \partial_4 q\big(i,v,(1,\mu),(g,\mu)\big)(g,B).$$
In this case, the conditions on the parameters are satisfied if 
$\sup \partial_2^j q(i,x,y,z) <\infty$ for $j=0,1,\dots,4$, $\sup \partial_3 q(i,x,y,z) <\infty$ and $\sup \partial_4 q(i,x,y,z) <\infty$, where the supremum is taken over all $i,x,y,z$. 

The limiting population density of \eqref{E:fSInfty} can be
decomposed into  two equations, for the limiting female and male measures:
\begin{align*}
(f_\F,\bar S^\F_t) &= (f_\F,\bar S^\F_0) + \int_0^t \big(f'_\F-h^\infty_{\bar S_u}(\F,\cdot)f_\F+f_\F(0)\mathfrak n^{\F,\infty}_{\bar S_u}, \bar S^\F_u\big) du; \\
(f_\M,\bar S^\M_t) &= (f_\M,\bar S^\M_0) + \int_0^t \big(f'_\M-h^\infty_{\bar S_u}(\M,m\cdot)f_\M, \bar S^\M_u\big) du + f_\M(0) \int_0^t \big(\mathfrak n^{\M,\infty}_{\bar S_u}, \bar S^\F_u\big) du. 
\end{align*}
In the case where these measures have densities (with respect to Lebesgue measure), 
namely $s^\F(v,t)$ for $\bar S^\F_t$ and $s^\M(v,t)$ for $\bar S^\M_t$, 
the densities resemble the McKendrick-von Foerster equations, 
as pointed out in \cite{JagKle16}: 
\begin{gather*}
\Big(\frac{\partial}{\partial t}+\frac{\partial}{\partial v}\Big) s^\F(v,t) = -h^\infty_{\bar S_t}(\F,v)s^\F(v,t), \q
s^\F(0,t) = \int_0^{t+a^*} \mathfrak n^{\F,\infty}_{\bar S_t}(v) s^\F(v,t) dv, \\
\Big(\frac{\partial}{\partial t}+\frac{\partial}{\partial v}\Big) s^\M(v,t) = -h^\infty_{\bar S_t}(\M,v)s^\M(v,t), \q
s^\M(0,t) = \int_0^{t+a^*} \mathfrak n^{\M,\infty}_{\bar S_t}(v) s^\F(v,t) dv.
\end{gather*}
The upper limit in these integrals is due to $s^\F(v,t)=0$ for $v>t+a^*$. 

The fluctuation limit \eqref{E:fZInfty} can be decomposed as follows,
with $Z^\F(dv) = Z(\{\F\},dv)$ and $Z^\M(dv) = Z(\{\M\},dv)$,
\begin{align*}
(f_\F,Z^\F_t) &= (f_\F,Z^\F_0)
+ \int_0^t \big( - \partial_S h^\infty_{\bar{S}_u}(Z_u)(\F,\cdot) f_\F + f_\F(0) \partial_S \mathfrak n^{\F,\infty}_{\bar{S}_u}(Z_u) , \bar{S}^\F_u \big) du \\
&\q\q\q+ \int_0^t \big( f'_\F - h^\infty_{\bar{S}_u}(\F,\cdot) f_\F + f_\F(0) \mathfrak n^{\F,\infty}_{\bar{S}_u} , Z^\F_u \big) du + \tilde{M}_t^\F,
\end{align*}
where
$$\big<\tilde M^\F\big>_t = \int_0^t \big(
f_\F(0)^2 \mathfrak w^{\F\F,\infty}_{\bar S_u} + h^\infty_{\bar S_u}(\F,\cdot) f_\F^2
- 2 f_\F(0) h^\infty_{\bar S_u}(\F,\cdot) \widehat{\mathfrak m}^{\F,\infty}_{\bar S_u} f_\F, \bar S^\F_u \big) du; $$
and
\begin{multline*}
(f_\M,Z^\M_t) = (f_\M,Z^\M_0)
+ \int_0^t \big( - \partial_S h^\infty_{\bar{S}_u}(Z_u)(\M,\cdot) f_\M, \bar{S}_u^\M \big) du \\
+ f_\M(0) \int_0^t \big(\partial_S \mathfrak n^{\M,\infty}_{\bar{S}_u}(Z_u) , \bar{S}_u^\F \big) du
+ \int_0^t \big( f'_\M - h^\infty_{\bar{S}_u}(\M,\cdot) f_\M, Z^\M_u \big) du \\
+ f_\M(0) \int_0^t \big(\mathfrak n^{\M,\infty}_{\bar{S}_u} , Z^\F_u \big) du + \tilde{M}_t^\M,
\end{multline*}
where
$$\big<\tilde M^\M\big>_t = \int_0^t \big( f_\M(0)^2 \mathfrak w^{\M\M,\infty}_{\bar S_u}, \bar S^\F_u \big) du
+ \int_0^t \big( h^\infty_{\bar S_u}(\M,\cdot) f_\M^2, \bar S^\M_u \big) du.$$

The McKendrick-von Foerster equations in their turn may take many special
forms, in different situations, dependent upon the varying
fertilisation, reproduction and survival patterns, as well as the role
of the carrying capacity type parameter. It is, however,
important that such modelling be done in close relationship to the
biological circumstances at hand. In this broad context, we contend
ourselves with the remark that typically female fertility will
increase with the abundance of males while mortality on the whole will tend to
increase with competition and, thus, population size. 
We shall not enter into the intricate details of this, like different forms of Allee effects about which there is a rich literature, e.g. \cite{EngLanSae03} and \cite{Gil04}. 

\subsection{With migration} \label{S:Migration}

Migration can be easily incorporated into the model. While emigrations can be seen and taken as deaths, immigrations can be formulated as individuals of different ages and types arriving at random times. Henceforth, immigrations will be our concern. In the case of controlled immigration, that is, when the rate of immigration and the number of immigrants per arrival are bounded, we can show that the limits for the Law of Large Numbers and Central Limit Theorem coincide with those without immigration. 

Suppose that immigrations occur at random times $\tilde\tau_1, \tilde\tau_2, \dots$ with rate $g(S_u,u)$, depending on the population composition and time. 
At time $\tilde\tau_j$, $\zeta_j$ new immigrants arrive with $\E[\zeta_j| \tilde\tau_j, S_{\tilde\tau_j}] = \tilde m(S_{\tilde\tau_j},\tilde\tau_j)$ and $\E[\zeta_j^2| \tilde\tau_j, S_{\tilde\tau_j}] = \tilde v(S_{\tilde\tau_j},\tilde\tau_j)$. 
The type and age of each immigrant is random, distributed according to some distribution $\tilde{\mathcal{K}}(\cdot; S_{\tilde\tau_j}, \tilde\tau_j)$ conditioning on the arrival time and the population structure at that time. 
In other words, if $x$ is an individual who immigrated at time $u$, denote its type and age at immigration by $\kappa_x$ and $\tilde\alpha_x$ respectively, then $\P(\kappa_x=i,\tilde\alpha_x\in A) = \int_{\{i\}\times A} \tilde{\mathcal{K}}(ds;S_u,u)$. 

Let $G((i,v),t)$ be the number of immigrants by time $t$ with type $i$ and age at immigration not more than $v$. Then, the equation analogous to \eqref{E:BasicEvol} is 
\begin{multline*}
(f,S_t) = (f,S_0) + \int_0^t (f',S_u) du + \sum_{i\in\K} f(i,0)B^i([0,t]) \\
- \int_{\S\times[0,t]} f(s) D(ds,du) + \int_{\S\times[0,t]} f(s) G(ds,du).
\end{multline*}
Compensating the above, we obtain a dynamical equation in the form of a semimartingale decomposition: 
\begin{multline*}
(f,S_t) = (f,S_0) + \int_0^t (L_{S_u}f,S_u) du \\ 
+ \int_0^t g(S_u,u) \tilde m(S_u,u) \int_{\S} f(s) \tilde{\mathcal{K}}(ds;S_u,u) du + \dddot M^f_t,
\end{multline*}
where $L_Sf$ is as in \eqref{E:fSt}
and $\dddot M^f_t$ is a locally square integrable martingale with predictable quadratic variation
\begin{align*}
\big<\dddot M^f\big>_t &= \int_0^t \bigg(
\sum_{i_1\in\K} \sum_{i_2\in\K} f(i_1,0)f(i_2,0) w^{i_1i_2}_{S_u} 
+ h_{S_u} f^2 - 2 \sum_{i\in\K} f(i,0) h_{S_u} \widehat m^i_{S_u} f , S_u \bigg) du \\
&\q\q+ \int_0^t g(S_u,u) \bigg\{ \tilde m(S_u,u) \int_\S f^2(s) \tilde{\mathcal{K}}(ds;S_u,u) \\
&\q\q\q\q\q\q\q\q+ \big(\tilde v(S_u,u) - \tilde m(S_u,u) \big) \Big(\int_S f(s) \tilde{\mathcal{K}}(ds;S_u,u) \Big)^2 \bigg\} du.
\end{align*}

Suppose that $g$, $\tilde m$ and $\tilde v$ are bounded. Then the Law of Large Numbers and the Central Limit Theorem hold, and the limits are the same as the corresponding limits in the case without immigration.

%%%%%%%%%%%%%%%%%%%%%%%%%%%%

\section{Sexual reproduction in serial monogamy mating system} \label{S:Monogamy}

Our framework can be applied immediately to model sexual reproduction, as in Sections \ref{S:ModelApplic1} and \ref{S:ModelApplic2}. 
However, the formulation sketched there does not take the form of mating into account. 
In this section, we adapt our framework to model a population with serial monogamy. 
In other words, a female and a male form a couple for life or for some time, like a breeding season.

\subsection{The model} \label{S:Monogamy-model}

Despite not restricting the model to human population, for the convenience of terminology, we refer to the formation of a couple as a ``marriage''. 
Consider a population consisting of individuals of three types; type 1 (denoted as {\Female}) refers to single females, type 2 (denoted as {\Male}) refers to single males and type 3 (denoted as {\FemaleMale}) refers to ``married'' couples. 
An individual from type {\F} and an individual from type {\M} can form a couple (get married) and becomes a type {\FM}. 
A couple lasts a period of random length interrupted by the death of one of the mates, in which case, the survivor becomes single and available for mating again. 

A type {\FM} gives birth at random times to random number of type {\F} and type {\M} offspring. 
We assume also the possibility of a type {\F} to give birth. 
The lifetime of each individual is random. 

Let {\FEMALE} (resp. {\MALE}) denote the set of all females (resp. males), single or married. 
Let $c_{x,y}$ denote the time individuals $x$ and $y$ become a couple, 
$c_x^j$ denote the time of the $j$th ``marriage'' of individual $x$ 
and $d_x^j$ denote the time $x$ becomes a ``widow'' from its $j$th marriage. 
To accomodate the notation for couples, the age of which consists of both that of the female partner and the male partner, we assign also an ``age'' of two indices to each individual that is single, with one index taking value $\infty$. 
Then, the age structure at time $t$ of the three types are given by 
\begin{align*}
S^\F_t(di,dv,dw) &= \sum_{x\in \FEMALE} \mathbf{1}_{\tau_x\le t\le \sigma_x} \big( \mathbf{1}_{t\le c_x^1} + \mathbf{1}_{d_x^1\le t\le c_x^2} + \mathbf{1}_{d_x^2\le t\le c_x^3} + \cdots \big) \delta_{(\F,t-\tau_x,\infty)}, \\
S^\M_t(di,dv,dw) &= \sum_{y\in \MALE} \mathbf{1}_{\tau_y\le t\le \sigma_y} \big( \mathbf{1}_{t\le c_y^1} + \mathbf{1}_{d_y^1\le t\le c_y^2} + \mathbf{1}_{d_y^2\le t\le c_y^3} + \cdots \big) \delta_{(\M,\infty,t-\tau_y)}, \\
S^\FM_t(di,dv,dw) &= \sum_{x\in \FEMALE} \sum_{y\in \MALE} \mathbf{1}_{\tau_x\le t\le \sigma_x} \mathbf{1}_{\tau_y\le t\le \sigma_y} \mathbf{1}_{c_{x,y}\le t} \delta_{(\FM,t-\tau_x,t-\tau_y)},
\end{align*}
and the structure of the entire population 
$$S_t(di,dv,dw) = S^\F_t(di,dv,dw) + S^\M_t(di,dv,dw) + S^\FM_t(di,dv,dw).$$

Let $b_S(v,w)$ denote the bearing rate of a couple with a female at age $v$ and a male at age $w$, when the population composition is $S$. 
At each birth, the number of females and the number of males born have the distributions of $\xi^\F_S(v,w)$ and $\xi^\M_S(v,w)$ respectively. Write $m^i_S(v,w) = \E[\xi^i_S(v,w)]$ and $\gamma^{ij}_S(v,w) = \E[\xi^i_S(v,w)\xi^j_S(v,w)]$. 
These quantities with $w=\infty$ correspond to those of single females. 

Let $\rho_S(v,w)$ be the rate at which a single female of age $v$ and a single male of age $w$ marry to each other. 
In other words, for $x\in\FEMALE$ and $y\in\MALE$, with $d_x^0 := 0 =: d_y^0$, 
$$\mathbf{1}_{c_{x,y}\le t} - \int_0^{t\wedge c_{x,y}} \rho_{S_u}(u-\tau_x,u-\tau_y) \mathbf{1}_{\tau_x\le u<\sigma_x} \mathbf{1}_{\tau_y\le u<\sigma_y} \sum_{j\ge0} \mathbf{1}_{d_x^j\le u\le c_x^{j+1}} \sum_{k\ge0} \mathbf{1}_{d_y^k\le u\le c_y^{k+1}} du$$
is a martingale. 
Suppose that marriage lasts for a random length, at rate $h^\FM_S(v,w)$ it breaks, dependent upon the ages of the mates as well as the population composition.

Let $h^\F_S(v,w)$ denote the death rate of a female when she is at age $v$ and her partner is at age $w$, when the population composition is $S$. 
Similarly, let $h^\M_S(v,w)$ denote the death rate of a male when he is at age $w$ and his partner is at age $v$, when the population composition is $S$. 
For single females and males, their death rates are $h^\F_S(v,\infty)$ and $h^\M_S(\infty,w)$, respectively. 

We consider test functions $f: \K \times \A\cup\{\infty\} \times \A\cup\{\infty\} \to \R$, 
with $\K = \{1,2,3\} \equiv \{\F,\M,\FM\}$,  
such that $f(3,v,w) \equiv f(\FM,v,w)$, $f(1,v,w) \equiv f(\F,v,w) = f(\F,v,\infty) = f_\F(v)$ and $f(2,v,w) \equiv f(\M,v,w) = f(\M,\infty,w) = f_\M(w)$, for any $v,w \in\A$, for some functions $f_\F$ and $f_\M$. 
We assume that $f(i,v,w)$ is differentiable with respect to $v$ and $w$, and $\partial_w f(\F,v,\infty) = \partial_v f(\M,\infty,w) = 0$. 

Then, it can be shown that (see Appendix \ref{S:Appendix-Monogamy-model}) $S$ has the following semimartingale representation\footnote{We write $((f(\ast,\Cdot),\mu)_{\Cdot},\nu)_\ast$ to mean $\int \int f(x,y)\mu(dy)\nu(dx) = (f,\nu\otimes\mu)$.}: 
\begin{align} \label{E:semimartB}
\notag&(f,S_t) = (f,S_0) + \int_0^t (\partial_v f + \partial_w f, S_u) du 
- \int_0^t (f h^\F_{S_u} \mathbf{1}_\F + f h^\M_{S_u} \mathbf{1}_\M, S_u) du \\
\notag&\q+ \int_0^t \bigg(\Big(\big(f(\M,\infty,\cdot)-f(\FM,\cdot,\cdot)\big) h^\F_{S_u} + \big(f(\F,\cdot,\infty)-f(\FM,\cdot,\cdot)\big) h^\M_{S_u} \Big) \mathbf{1}_\FM, S_u\bigg) du \\
\notag&\q+ \int_0^t \Big(\big(f(\F,\cdot,\infty)+f(\M,\infty,\cdot)-f(\FM,\cdot,\cdot)\big) h^\FM_{S_u} \mathbf{1}_\FM, S_u\Big) du \\
\notag&\q+ \int_0^t ( [f(\F,0,\infty)b_{S_u}m^\F_{S_u} + f(\M,\infty,0)b_{S_u}m^\M_{S_u}] (\mathbf{1}_\FM+\mathbf{1}_\F), S_u) du \\
\notag&\q- \int_0^t (( [f(\F,\Cdot,\infty)+f(\M,\infty,\ast)-f(\FM,\Cdot,\ast)] \rho_{S_u}(\Cdot,\ast) \mathbf{1}_\F, S_u)_{\Cdot} \mathbf{1}_\M, S_u)_\ast du \\
&\q+ M^f_t,
\end{align}
where $M^f$ is a martingale with predictable quadratic variation 
\begin{align*}
&\left<M^f\right>_t 
= \int_0^t (f^2 h^\F_{S_u} \mathbf{1}_\F + f^2 h^\M_{S_u} \mathbf{1}_\M, S_u) du \\
&\q+ \int_0^t \bigg(\Big( [f(\F,\cdot,\infty)-f(\FM,\cdot,\cdot)]^2 h^\M_{S_u} + [f(\M,\infty,\cdot)-f(\FM,\cdot,\cdot)]^2 h^\F_{S_u} \Big) \mathbf{1}_\FM, S_u\bigg) du\\
&\q+ \int_0^t \Big(\big(f(\F,\cdot,\infty)+f(\M,\infty,\cdot)-f(\FM,\cdot,\cdot)\big)^2 h^\FM_{S_u} \mathbf{1}_\FM, S_u\Big) du \\
&\q+ \int_0^t ( [f^2(\F,0,\infty)\gamma^{\F\F}_{S_u} + f^2(\M,\infty,0)\gamma^{\M\M}_{S_u} + 2 f(\F,0,\infty)f(\M,\infty,0)\gamma^{\F\M}_{S_u}] b_{S_u} (\mathbf{1}_\FM+\mathbf{1}_\F), S_u) du \\
&\q+ \int_0^t (( [f(\F,\Cdot,\infty)+f(\M,\infty,\ast)-f(\FM,\Cdot,\ast)]^2 \rho_{S_u}(\Cdot,\ast) \mathbf{1}_\F, S_u)_{\Cdot} \mathbf{1}_\M, S_u)_\ast du. 
\end{align*}

Considering a family of the population processes indexed by $K\ge1$ and writing the reproduction parameters in the form $q^K_{S/K}(\cdot)$, we obtain the asymptotics as $K\to\infty$. 

\subsection{Law of Large Numbers}

Let $\bar S^K_t = S^K_t/K$. Then the scaled process satisfies 
\begin{align*}
&(f,\bar S^K_t) = (f,\bar S^K_0) + \int_0^t \big(\partial_v f + \partial_w f, \bar S^K_u\big) du 
- \int_0^t \big(f h^{\F,K}_{\bar S^K_u} \mathbf{1}_\F + f h^{\M,K}_{\bar S^K_u} \mathbf{1}_\M, \bar S^K_u\big) du \\
&\q+ \int_0^t \Big(\Big[\big(f(\M,\infty,\cdot)-f(\FM,\cdot,\cdot)\big) h^{\F,K}_{\bar S^K_u} + \big(f(\F,\cdot,\infty)-f(\FM,\cdot,\cdot)\big) h^{\M,K}_{\bar S^K_u} \Big] \mathbf{1}_\FM, \bar S^K_u\Big) du \\
&\q+ \int_0^t \Big(\big(f(\F,\cdot,\infty)+f(\M,\infty,\cdot)-f(\FM,\cdot,\cdot)\big) h^{\FM,K}_{\bar S^K_u} \mathbf{1}_\FM, \bar S^K_u\Big) du \\
&\q+ \int_0^t \Big( \big[f(\F,0,\infty)m^{\F,K}_{\bar S^K_u} + f(\M,\infty,0)m^{\M,K}_{\bar S^K_u}\big] b^K_{\bar S^K_u} (\mathbf{1}_\FM+\mathbf{1}_\F), \bar S^K_u\Big) du \\
&\q- \int_0^t \Big(\Big( [f(\F,\Cdot,\infty)+f(\M,\infty,\ast)-f(\FM,\Cdot,\ast)] K \rho^K_{\bar S^K_u}(\Cdot,\ast) \mathbf{1}_\F, \bar S^K_u\Big)_{\Cdot} \mathbf{1}_\M, \bar S^K_u\Big)_\ast du \\
&\q+ \bar M^{f,K}_t,
\end{align*}
with
\begin{align*}
&\left<\bar M^{f,K}\right>_t = \int_0^t \frac1K \big(f^2 h^{\F,K}_{\bar S^K_u} \mathbf{1}_\F + f^2 h^{\M,K}_{\bar S^K_u} \mathbf{1}_\M, \bar S^K_u\big) du \\
&\q+ \int_0^t \frac1K \Big(\Big[ \big(f(\M,\infty,\cdot)-f(\FM,\cdot,\cdot)\big)^2 h^{\F,K}_{\bar S^K_u} + \big(f(\F,\cdot,\infty)-f(\FM,\cdot,\cdot)\big)^2 h^{\M,K}_{\bar S^K_u} \Big] \mathbf{1}_\FM, \bar S^K_u\Big) du\\
&\q+ \int_0^t \frac1K \Big(\big(f(\F,\cdot,\infty)+f(\M,\infty,\cdot)-f(\FM,\cdot,\cdot)\big)^2 h^{\FM,K}_{\bar S^K_u} \mathbf{1}_\FM, \bar S^K_u\Big) du\\
&\q+ \int_0^t \frac1K \Big( \big[f^2(\F,0,\infty)\gamma^{\F\F,K}_{\bar S^K_u} + f^2(\M,\infty,0)\gamma^{\M\M,K}_{\bar S^K_u} + 2 f(\F,0,\infty)f(\M,\infty,0)\gamma^{\F\M,K}_{\bar S^K_u}\big] b^K_{\bar S^K_u} (\mathbf{1}_\FM+\mathbf{1}_\F), \bar S^K_u\Big) du \\
&\q+ \int_0^t \big(\big( [f(\F,\Cdot,\infty)+f(\M,\infty,\ast)-f(\FM,\Cdot,\ast)]^2 \rho^K_{\bar S^K_u}(\Cdot,\ast) \mathbf{1}_\F, \bar S^K_u\big)_{\Cdot} \mathbf{1}_\M, \bar S^K_u\big)_\ast du. 
\end{align*}

For the convergence of the sequence of processes $\bar S^K$, we need conditions like (C0) to (C3), with conditions on $\rho$ too. 

\begin{enumerate}
\item [(C0')]
In addition to (C0), $K\rho^K$ is bounded. 
\item [(C1')]
(C1) holds also for $q^K$ being $K\rho^K$.
\item [(C2')]
(C2) holds, and $\lim_{K\to\infty} K\rho^K_\mu =: \rho^\infty_\mu$. 
\item [(C3')]
Same as (C3).
\end{enumerate}

\begin{theorem} \label{T:MonogamyLLN}
Under the smooth demography conditions (C0')-(C3'), 
the scaled process $\bar S^K$ converges weakly in the Skorokhod space $\D(\T,\mathcal{M}(\S))$, as $K\to\infty$, to a deterministic measure-valued process $\bar S$, satisfying
\begin{align}
\notag&(f,\bar S_t) = (f,\bar S_0) + \int_0^t \big(\partial_v f + \partial_w f, \bar S_u\big) du 
- \int_0^t \big(f h^{\F,\infty}_{\bar S_u} \mathbf{1}_\F + f h^{\M,\infty}_{\bar S_u} \mathbf{1}_\M, \bar S_u\big) du \\
\notag&\q+ \int_0^t \Big(\Big[\big(f(\M,\infty,\cdot)-f(\FM,\cdot,\cdot)\big) h^{\F,\infty}_{\bar S_u} + \big(f(\F,\cdot,\infty)-f(\FM,\cdot,\cdot)\big) h^{\M,\infty}_{\bar S_u} \Big] \mathbf{1}_\FM, \bar S_u\Big) du \\
\notag&\q+ \int_0^t \Big(\big(f(\F,\cdot,\infty)+f(\M,\infty,\cdot)-f(\FM,\cdot,\cdot)\big) h^{\FM,\infty}_{\bar S_u} \mathbf{1}_\FM, \bar S_u\Big) du \\
\notag&\q+ \int_0^t \Big( \big[f(\F,0,\infty)m^{\F,\infty}_{\bar S_u} + f(\M,\infty,0)m^{\M,\infty}_{\bar S_u}\big] b^\infty_{\bar S_u} (\mathbf{1}_\FM+\mathbf{1}_\F), \bar S_u\Big) du \\
&\q- \int_0^t \Big(\Big( \big[f(\F,\Cdot,\infty)+f(\M,\infty,\ast)-f(\FM,\Cdot,\ast)\big] \rho^\infty_{\bar S_u}(\Cdot,\ast) \mathbf{1}_\F, \bar S_u\Big)_{\Cdot} \mathbf{1}_\M, \bar S_u\Big)_\ast du. \label{E:SbarInftyMonogamy}
\end{align}
\end{theorem}

From this, we obtain a system of PDEs for the age densities of the three subpopulations. 
Suppose $a^\F(v,t)$, $a^\M(w,t)$ and $a^\FM(v,w,t)$ are the densities (in $v$ and $w$) of $\bar S^i_t$, $i=\F,\M,\FM$, respectively. Then the densities satisfy the following: 
\begin{align*}
&\partial_t a^\F(v,t) + \partial_v a^\F(v,t) = 
- h^{\F,\infty}_{\bar S_t}(v,\infty) a^\F(v,t) \\
&\q\q+ \int \big(h^{\M,\infty}_{\bar S_t}(v,w) + h^{\FM,\infty}_{\bar S_t}(v,w)\big) a^\FM(v,w,t) dw 
- \int \rho^\infty_{\bar S_t}(v,w) a^\M(w,t) dw a^\F(v,t), \\
&\partial_t a^\M(w,t) + \partial_w a^\M(w,t) = 
- h^{\M,\infty}_{\bar S_t}(\infty,w) a^\M(w,t) \\
&\q\q+ \int \big(h^{\F,\infty}_{\bar S_t}(v,w) + h^{\FM,\infty}_{\bar S_t}(v,w)\big) a^\FM(v,w,t) dv 
- \int \rho^\infty_{\bar S_t}(v,w) a^\F(v,t) dv a^\M(w,t), \\
&\partial_t a^\FM(v,w,t) + \partial_v a^\FM(v,w,t) + \partial_w a^\FM(v,w,t) = \\
&\q\q- \big(h^{\F,\infty}_{\bar S_t}(v,w) + h^{\M,\infty}_{\bar S_t}(v,w) + h^{\FM,\infty}_{\bar S_t}(v,w)\big) a^\FM(v,w,t) 
+ \rho^\infty_{\bar S_t}(v,w) a^\F(v,t)a^\M(w,t), 
\end{align*}
with boundary conditions
\begin{align*}
a^\F(0,t) &= \int\int m^{\F,\infty}_{\bar S_t}(v,w) b^\infty_{\bar S_t}(v,w) a^\FM(v,w,t) dvdw + \int m^{\F,\infty}_{\bar S_t}(v,\infty) b^\infty_{\bar S_t}(v,\infty) a^\F(v,t) dv, \\
a^\M(0,t) &= \int\int m^{\M,\infty}_{\bar S_t}(v,w) b^\infty_{\bar S_t}(v,w) a^\FM(v,w,t) dvdw + \int m^{\M,\infty}_{\bar S_t}(v,\infty) b^\infty_{\bar S_t}(v,\infty) a^\F(v,t) dv.
\end{align*}
This is comparable to that given by Fredrickson in \cite{Fre71}. 

\subsection{Central Limit Theorem}

Let $Z^K = \sqrt{K}(\bar S^K - \bar S)$.
Then, for any $f\in C^1$ and $t\in\T$, 
\begin{align} 
\notag&(f,Z_t^K) = (f,Z_0^K) + \int_0^t \big(\partial_v f + \partial_w f, Z^K_u\big) du \\
\notag&\q- \int_0^t \sqrt{K} \big(f (h^{\F,K}_{\bar S^K_u}-h^{\F,\infty}_{\bar S_u}) \mathbf{1}_\F + f (h^{\M,K}_{\bar S^K_u}-h^{\M,\infty}_{\bar S_u}) \mathbf{1}_\M, \bar S_u\big) du 
- \int_0^t \big(f h^{\F,K}_{\bar S^K_u} \mathbf{1}_\F + f h^{\M,K}_{\bar S^K_u} \mathbf{1}_\M, Z^K_u\big) du \\
\notag&\q+ \int_0^t \sqrt{K} \Big(\Big[\big(f(\M,\infty,\cdot)-f(\FM,\cdot,\cdot)\big) (h^{\F,K}_{\bar S^K_u}-h^{\F,\infty}_{\bar S_u}) + \big(f(\F,\cdot,\infty)-f(\FM,\cdot,\cdot)\big) (h^{\M,K}_{\bar S^K_u}-h^{\M,\infty}_{\bar S_u}) \Big] \mathbf{1}_\FM, \bar S_u\Big) du \\
\notag&\q+ \int_0^t \Big(\Big[\big(f(\M,\infty,\cdot)-f(\FM,\cdot,\cdot)\big) h^{\F,K}_{\bar S^K_u} + \big(f(\F,\cdot,\infty)-f(\FM,\cdot,\cdot)\big) h^{\M,K}_{\bar S^K_u} \Big] \mathbf{1}_\FM, Z^K_u\Big) du \\
\notag&\q+ \int_0^t \sqrt{K} \Big(\big(f(\F,\cdot,\infty)+f(\M,\infty,\cdot)-f(\FM,\cdot,\cdot)\big) (h^{\FM,K}_{\bar S^K_u}-h^{\FM,\infty}_{\bar S_u}) \mathbf{1}_\FM, \bar S_u\Big) du \\
\notag&\q+ \int_0^t \Big(\big(f(\F,\cdot,\infty)+f(\M,\infty,\cdot)-f(\FM,\cdot,\cdot)\big) h^{\FM,K}_{\bar S^K_u} \mathbf{1}_\FM, Z^K_u\Big) du \\
\notag&\q+ \int_0^t \sqrt{K} \Big( \big[f(\F,0,\infty) (b^K_{\bar S^K_u}m^{\F,K}_{\bar S^K_u} - b^\infty_{\bar S_u}m^{\F,\infty}_{\bar S_u}) + f(\M,\infty,0) (b^K_{\bar S^K_u}m^{\M,K}_{\bar S^K_u} - b^\infty_{\bar S_u}m^{\M,\infty}_{\bar S_u}) \big] (\mathbf{1}_\FM+\mathbf{1}_\F), \bar S_u\Big) du \\
\notag&\q+ \int_0^t \Big( \big[f(\F,0,\infty)m^{\F,K}_{\bar S^K_u} + f(\M,\infty,0)m^{\M,K}_{\bar S^K_u}\big] b^K_{\bar S^K_u} (\mathbf{1}_\FM+\mathbf{1}_\F), Z^K_u\Big) du \\
\notag&\q- \int_0^t \sqrt{K} \Big(\Big( [f(\F,\Cdot,\infty)+f(\M,\infty,\ast)-f(\FM,\Cdot,\ast)] (K \rho^K_{\bar S^K_u}(\Cdot,\ast) - \rho^\infty_{\bar S_u}(\Cdot,\ast)) \mathbf{1}_\F, \bar S_u\Big)_{\Cdot} \mathbf{1}_\M, \bar S_u\Big)_\ast du \\
\notag&\q- \int_0^t \Big(\Big( [f(\F,\Cdot,\infty)+f(\M,\infty,\ast)-f(\FM,\Cdot,\ast)] K \rho^K_{\bar S^K_u}(\Cdot,\ast) \mathbf{1}_\F, Z^K_u\Big)_{\Cdot} \mathbf{1}_\M, \bar S_u\Big)_\ast du \\
\notag&\q- \int_0^t \Big(\Big( [f(\F,\Cdot,\infty)+f(\M,\infty,\ast)-f(\FM,\Cdot,\ast)] K \rho^K_{\bar S^K_u}(\Cdot,\ast) \mathbf{1}_\F, \bar S_u\Big)_{\Cdot} \mathbf{1}_\M, Z^K_u\Big)_\ast du \\
&\q+ \tilde M^{f,K}_t, \label{E:fZKMonogamy}
\end{align}
where $\tilde{M}_t^{f,K}$ is a square integrable martingale with predictable quadratic variation
\begin{align*}
&\left<\tilde M^{f,K}\right>_t 
= \int_0^t \big(f^2 h^{\F,K}_{\bar S^K_u} \mathbf{1}_\F + f^2 h^{\M,K}_{\bar S^K_u} \mathbf{1}_\M, \bar S^K_u\big) du \\
&\q+ \int_0^t \Big(\Big[ \big(f(\M,\infty,\cdot)-f(\FM,\cdot,\cdot)\big)^2 h^{\F,K}_{\bar S^K_u} + \big(f(\F,\cdot,\infty)-f(\FM,\cdot,\cdot)\big)^2 h^{\M,K}_{\bar S^K_u} \Big] \mathbf{1}_\FM, \bar S^K_u\Big) du\\
&\q+ \int_0^t \Big(\big(f(\F,\cdot,\infty)+f(\M,\infty,\cdot)-f(\FM,\cdot,\cdot)\big)^2 h^{\FM,K}_{\bar S^K_u} \mathbf{1}_\FM, \bar S^K_u\Big) du\\
&\q+ \int_0^t \Big( \big[f^2(\F,0,\infty)\gamma^{\F\F,K}_{\bar S^K_u} + f^2(\M,\infty,0)\gamma^{\M\M,K}_{\bar S^K_u} + 2 f(\F,0,\infty)f(\M,\infty,0)\gamma^{\F\M,K}_{\bar S^K_u}\big] b^K_{\bar S^K_u} (\mathbf{1}_\FM+\mathbf{1}_\F), \bar S^K_u\Big) du \\
&\q+ \int_0^t \big(\big( [f(\F,\Cdot,\infty)+f(\M,\infty,\ast)-f(\FM,\Cdot,\ast)]^2 K\rho^K_{\bar S^K_u}(\Cdot,\ast) \mathbf{1}_\F, \bar S^K_u\big)_{\Cdot} \mathbf{1}_\M, \bar S^K_u\big)_\ast du. 
\end{align*}

For the Central Limit Theorem, we need to include the additional parameter $\rho$ to conditions (A0)--(A4). 

\begin{enumerate}
\item[(A0')]
Same as (A0).
\item[(A1')]
Same as (A1).
\item[(A2')]
In addition to (A2), the same holds for $q^K$ being $K\rho^K$ and $q^\infty$ being $\rho^\infty$. 
\item[(A3')]
(A3) holds for $q=b,h,m,\rho$. 
\item[(A4')]
Same as (A4).
\end{enumerate}

\begin{theorem} \label{T:CLTMonogamy}
Under the assumptions (C0')--(C3') and (A0')--(A4'), the process $(Z^K_t)_{t\in\T}$ converges weakly in $\mathbb{D}(\T,W^{-4})$ as $K\to\infty$ to the process $(Z_t)_{t\in\T}$ that satisfies, for $f\in W^4$ and $t\in\T$, 
\begin{align}
\notag&(f,Z_t) = (f,Z_0) + \int_0^t \big(\partial_v f + \partial_w f, Z_u\big) du \\
\notag&\q- \int_0^t \big(f \partial_Sh^{\F,\infty}_{\bar S_u}(Z_u) \mathbf{1}_\F + f \partial_Sh^{\M,\infty}_{\bar S_u}(Z_u) \mathbf{1}_\M, \bar S_u\big) du 
- \int_0^t \big(f h^{\F,\infty}_{\bar S_u} \mathbf{1}_\F + f h^{\M,\infty}_{\bar S_u} \mathbf{1}_\M, Z_u\big) du \\
\notag&\q+ \int_0^t \Big(\Big[\big(f(\M,\infty,\cdot)-f(\FM,\cdot,\cdot)\big) \partial_Sh^{\F,\infty}_{\bar S_u}(Z_u) + \big(f(\F,\cdot,\infty)-f(\FM,\cdot,\cdot)\big) \partial_Sh^{\M,\infty}_{\bar S_u}(Z_u) \Big] \mathbf{1}_\FM, \bar S_u\Big) du \\
\notag&\q+ \int_0^t \Big(\Big[\big(f(\M,\infty,\cdot)-f(\FM,\cdot,\cdot)\big) h^{\F,\infty}_{\bar S_u} + \big(f(\F,\cdot,\infty)-f(\FM,\cdot,\cdot)\big) h^{\M,\infty}_{\bar S_u} \Big] \mathbf{1}_\FM, Z_u\Big) du \\
\notag&\q+ \int_0^t \Big(\big(f(\F,\cdot,\infty)+f(\M,\infty,\cdot)-f(\FM,\cdot,\cdot)\big) \partial_Sh^{\FM,\infty}_{\bar S_u}(Z_u) \mathbf{1}_\FM, \bar S_u\Big) du \\
\notag&\q+ \int_0^t \Big(\big(f(\F,\cdot,\infty)+f(\M,\infty,\cdot)-f(\FM,\cdot,\cdot)\big) h^{\FM,\infty}_{\bar S_u} \mathbf{1}_\FM, Z_u\Big) du \\
\notag&\q+ \int_0^t \Big( \big[f(\F,0,\infty) \partial_S(b^\infty_{\bar S_u}m^{\F,\infty}_{\bar S_u})(Z_u) + f(\M,\infty,0) \partial_S(b^\infty_{\bar S_u}m^{\M,\infty}_{\bar S_u})(Z_u) \big] (\mathbf{1}_\FM+\mathbf{1}_\F), \bar S_u\Big) du \\
\notag&\q+ \int_0^t \Big( \big[f(\F,0,\infty)m^{\F,\infty}_{\bar S_u} + f(\M,\infty,0)m^{\M,\infty}_{\bar S_u}\big] b^\infty_{\bar S_u}(\mathbf{1}_\FM+\mathbf{1}_\F), Z_u\Big) du \\
\notag&\q- \int_0^t \Big(\Big( [f(\F,\Cdot,\infty)+f(\M,\infty,\ast)-f(\FM,\Cdot,\ast)] \partial_S\rho^\infty_{\bar S_u}(Z_u)(\Cdot,\ast) \mathbf{1}_\F, \bar S_u\Big)_{\Cdot} \mathbf{1}_\M, \bar S_u\Big)_\ast du \\
\notag&\q- \int_0^t \Big(\Big( [f(\F,\Cdot,\infty)+f(\M,\infty,\ast)-f(\FM,\Cdot,\ast)] \rho^\infty_{\bar S_u}(\Cdot,\ast) \mathbf{1}_\F, Z_u\Big)_{\Cdot} \mathbf{1}_\M, \bar S_u\Big)_\ast du \\
\notag&\q- \int_0^t \Big(\Big( [f(\F,\Cdot,\infty)+f(\M,\infty,\ast)-f(\FM,\Cdot,\ast)] \rho^\infty_{\bar S_u}(\Cdot,\ast) \mathbf{1}_\F, \bar S_u\Big)_{\Cdot} \mathbf{1}_\M, Z_u\Big)_\ast du \\
&\q+ \tilde M^{f,\infty}_t,
\label{E:fZinftyMonogamy}
\end{align}
where $\tilde M^{f,\infty}$ is a continuous Gaussian martingale with predictable quadratic variation 
\begin{align*}
&\left<\tilde M^{f,\infty}\right>_t 
= \int_0^t \big(f^2 h^{\F,\infty}_{\bar S_u} \mathbf{1}_\F + f^2 h^{\M,\infty}_{\bar S_u} \mathbf{1}_\M, \bar S_u\big) du \\
&\q+ \int_0^t \Big(\Big[ \big(f(\M,\infty,\cdot)-f(\FM,\cdot,\cdot)\big)^2 h^{\F,\infty}_{\bar S_u} + \big(f(\F,\cdot,\infty)-f(\FM,\cdot,\cdot)\big)^2 h^{\M,\infty}_{\bar S_u} \Big] \mathbf{1}_\FM, \bar S_u\Big) du\\
&\q+ \int_0^t \Big(\big(f(\F,\cdot,\infty)+f(\M,\infty,\cdot)-f(\FM,\cdot,\cdot)\big)^2 h^{\FM,\infty}_{\bar S_u} \mathbf{1}_\FM, \bar S_u\Big) du\\
&\q+ \int_0^t \Big( \big[f^2(\F,0,\infty)\gamma^{\F\F,\infty}_{\bar S_u} + f^2(\M,\infty,0)\gamma^{\M\M,\infty}_{\bar S_u} + 2 f(\F,0,\infty)f(\M,\infty,0)\gamma^{\F\M,\infty}_{\bar S_u}\big] b^\infty_{\bar S_u} (\mathbf{1}_\FM+\mathbf{1}_\F), \bar S_u\Big) du \\
&\q+ \int_0^t \big(\big( [f(\F,\Cdot,\infty)+f(\M,\infty,\ast)-f(\FM,\Cdot,\ast)]^2 \rho^\infty_{\bar S_u}(\Cdot,\ast) \mathbf{1}_\F, \bar S_u\big)_{\Cdot} \mathbf{1}_\M, \bar S_u\big)_\ast du. 
\end{align*}
\end{theorem}

%%%%%%%%%%%%%%%%%%%%%%%%%%%%

\appendix

\section{Appendices for Section \ref{S:Multitype}} \label{S:Appendix-Multitype}

Here we include the proofs of the results of this paper. 
We will use $c$, with or without subscript, to denote generic constants, all independent of $K$. 

\subsection{Proofs on the model set-up} \label{SetupProof}

\subsubsection{A generator formulation}

A model description relying on the generator of the process was given in \cite{JagKle16}. 
We restate it here for completeness. 
For $F \in C_b^1(\R)$ and $f \in C^1(\S)$, the limit
$$\lim_{t\to 0} \frac1t \E_S\big[F((f,S_t))-F((f,S))\big] = \mathcal{G}F((f,S))$$
exists and
\begin{multline*}
\mathcal{G}F((f,S)) = F'((f,S))(f',S) \\
+ \sum_{j=1}^{(1,S)} b_S(s_j) \Bigg( \E\bigg[ F\bigg( \sum_{i\in\K} f(i,0)\widecheck \xi_S^i(s_j) + (f,S) \bigg) \bigg] - F\big((f,S)\big) \Bigg) \\
+ \sum_{j=1}^{(1,S)} h_S(s_j) \Bigg( \E\bigg[ F\bigg( \sum_{i\in\K} f(i,0)\widehat \xi_S^i(s_j) + (f,S) - f(s_j) \bigg) \bigg] - F\big((f,S)\big) \Bigg),
\end{multline*}
where $(1,S)$ is the size of $S$. 
Consequently, Dynkin's formula holds,
$$F((f,S_t)) = F((f,S_0)) + \int_0^t \mathcal{G}F((f,S_u)) du + M_t^{F,f},$$
where $M^{F,f}_t$ is a local martingale with the predictable quadratic variation
$$\big<M^{F,f}\big>_t = \int_0^t \big( \mathcal{G}F^2((f,S_u)) - 2F((f,S_u))\mathcal{G}F((f,S_u)) \big) du.$$

Writing $M_t^{f}$ for the special case where $F$ is chosen as the identity function mapping $u$ into itself, we obtain \eqref{E:fSt}. 

\subsubsection{A formulation through evolution} \label{S:AppA-model}

The model can also be described by analysis of the evolution of the population as done in \cite{FanEtal19} for single-type case.
Recall the representation of $S_t$ in \eqref{E:S}. 
Let $S^{(i)}_t(dv) = S_t(\{i\},dv) = \sum_{x\in I} \mathbf{1}_{\tau_x\le t<\sigma_x} \mathbf{1}_{\kappa_x=i} \delta_{t-\tau_x}(dv)$ be the age structure of the subpopulation of type $i$ individuals. Then, $(f,S_t) = \sum_{i\in\K} (f(i,\cdot),S^{(i)}_t)$.
From \cite{FanEtal19}, we have, for $f \in C^1(\S)$,
\begin{multline*}
(f(i,\cdot),S^{(i)}_t) = (f(i,\cdot),S^{(i)}_0) + \int_0^t (f'(i,\cdot),S^{(i)}_u) du \\
+ f(i,0)B^i([0,t]) - \int_{\A\times[0,t]} f(i,v) D^i(dv,du),
\end{multline*}
where $B^i(t)$ is the number of individuals of type $i$ born by time $t$ and $D^i(v,t)$ is the number of individuals of type $i$ who died by time $t$ and whose life span was not greater than $v$: 
$$B^i(t) = \sum_{x\in I} \mathbf{1}_{\kappa_x=i} \mathbf{1}_{\tau_x\le t}
\q \textnormal{and} \q
D^i(v,t) = \sum_{x\in I} \mathbf{1}_{\kappa_x=i} \mathbf{1}_{\lambda_x\le v} \mathbf{1}_{\sigma_x\le t}.$$
Therefore, $S_t$ satisfies the following:
for $f \in C^1(\S)$,
\begin{equation} \label{E:BasicEvol}
(f,S_t) = (f,S_0) + \int_0^t (f',S_u) du + \sum_{i\in\K} f(i,0)B^i([0,t]) - \int_{\S\times[0,t]} f(s) D(ds,du),
\end{equation}
where $D((i,v),t) = D^i(v,t)$. 
We obtain \eqref{E:fSt} by compensating the last two terms in \eqref{E:BasicEvol}. 

For each $i$, $M_{\widecheck B^i,f}(t) := f(i,0) \big(\widecheck B^i([0,t])-\int_0^t(b_{S_u}\widecheck m^i_{S_u},S_u)du\big)$ is a martingale with predictable quadratic variation $\big<M_{\widecheck B^i,f}\big>_t = f^2(i,0) \int_0^t (b_{S_u}\widecheck \gamma^{ii}_{S_u},S_u) du$, with $\widecheck B$ denoting the number of individuals born through bearings of mothers by time $t$.
Sum over $i$ to obtain the compensated process 
$$M_{\widecheck B,f}(t) := \sum_{i\in\K} f(i,0) \bigg(\widecheck B^i([0,t])-\int_0^t(b_{S_u}\widecheck m^i_{S_u},S_u)du\bigg)$$
as a martingale with the predictable quadratic variation 
\begin{equation} \label{E:MhatBQV}
\big<M_{\widecheck B,f}\big>_t = \sum_{i_1\in\K} \sum_{i_2\in\K} f(i_1,0) f(i_2,0) \int_0^t (b_{S_u}\widecheck \gamma^{i_1i_2}_{S_u},S_u) du.
\end{equation}
Note that $\big[M_{\widecheck B^{i_1},f}, M_{\widecheck B^{i_2},f}\big]_t = \sum_{u\le t} \Delta M_{\widecheck B^{i_1},f}(u) \Delta M_{\widecheck B^{i_2},f}(u)$.
For $i_1\neq i_2$, the product of jumps $\Delta M_{\widecheck B^{i_1},f}(u) \Delta M_{\widecheck B^{i_2},f}(u)$
is non-zero precisely when there is a birth of both types, in which case it is equal to
$$%\Delta M_{\widecheck B^{i_1},f}(u) \Delta M_{\widecheck B^{i_2},f}(u) =
\sum_{x\in I} f(i_1,0)f(i_2,0) \mathbf{1}_{\sigma_x>u}
\Big(\sum_{j\in\N} \mathbf{1}_{\tau_{xj}=u} \mathbf{1}_{\kappa_{xj}=i_1}\Big)
\Big(\sum_{j\in\N} \mathbf{1}_{\tau_{xj}=u} \mathbf{1}_{\kappa_{xj}=i_2}\Big).$$ 
Thus, 
$$\big<M_{\widecheck B^{i_1},f}, M_{\widecheck B^{i_2},f}\big>_t = f(i_1,0)f(i_2,0) \int_0^t (b_{S_u}\widecheck \gamma^{i_1,i_2}_{S_u},S_u) du$$
is a compensator of $\big[M_{\widecheck B^{i_1},f}, M_{\widecheck B^{i_2},f}\big]_t$, and, since 
$$\big<M_{\widecheck B,f}\big>_t = \sum_{i_1\in\K}\sum_{i_2\in\K} \big<M_{\widecheck B^{i_1},f}, M_{\widecheck B^{i_2},f}\big>_t,$$ \eqref{E:MhatBQV} follows.
In a similar manner, with $\widehat B$ denoting the number of individuals generated through splitting %(at the deaths of mothers) 
by time $t$, 
$M_{\widehat B,f}(t) := \sum_{i\in\K} f(i,0) \big(\widehat B^i([0,t])-\int_0^t(h_{S_u}\widehat m^i_{S_u},S_u)du\big)$
is a martingale with
$$\big<M_{\widehat B,f}\big>_t = \sum_{i_1\in\K} \sum_{i_2\in\K} f(i_1,0) f(i_2,0) \int_0^t (h_{S_u}\widehat \gamma^{i_1i_2}_{S_u},S_u) du,$$
and
$M_{D,f}(t) := \int_{\S\times[0,t]} f(s) D(ds,du) - \int_0^t(h_{S_u}f,S_u)du$
is a martingale with
$$\big<M_{D,f}\big>_t = \int_0^t (h_{S_u}f^2,S_u) du.$$
Now, let $M^f_t = M_{\widecheck B,f}(t) + M_{\widehat B,f}(t) - M_{D,f}(t)$.
Analysing cross terms as before and adding them together, we have, since $\big<M_{\widecheck B,f},M_{\widehat B,f}\big>_t =0$ and $\big<M_{\widecheck B,f},M_{D,f}\big>_t=0$, that
\begin{align*}
&\big<M^f\big>_t = \big<M_{\widecheck B,f}\big>_t + \big<M_{\widehat B,f}\big>_t + \big<M_{D,f}\big>_t - 2\big<M_{\widehat B,f},M_{D,f}\big>_t \\
&= \int_0^t \bigg(
\sum_{i_1\in\K} \sum_{i_2\in\K} f(i_1,0)f(i_2,0) w^{i_1i_2}_{S_u} + h_{S_u} f^2
- 2 \sum_{i\in\K} f(i,0) h_{S_u} \widehat m^{i}_{S_u} f, S_u \bigg) du.
\end{align*}
In fact, it can further be shown that, for $f,g$ on $\S$,
\begin{multline*}
\big<M^f,M^g\big>_t = \int_0^t \bigg(
\sum_{i_1\in\K} \sum_{i_2\in\K} f(i_1,0)g(i_2,0) w^{i_1i_2}_{S_u} + h_{S_u} fg \\
- \sum_{i\in\K} \big(f(i,0)g+g(i,0)f\big) h_{S_u} \widehat m^{i}_{S_u}, S_u \bigg) du.
\end{multline*}

\subsubsection{Proof of Remark \ref{R:MeasureM}} \label{S:MeasureMProof}

Let
\begin{align*}
M_t(ds)
&= \sum_{i\in\K} \delta_{(i,0)}(ds) \bigg(\widecheck B^i([0,t]) - \int_0^t (b_{S_u}\widecheck m^i_{S_u}, S_u) du\bigg) \\
&\q\q+ \sum_{i\in\K} \delta_{(i,0)}(ds) \bigg(\widehat B^i([0,t]) - \int_0^t (h_{S_u}\widehat m^i_{S_u}, S_u) du\bigg) \\
&- \bigg(\sum_{x\in I} \delta_{(\kappa_x,\lambda_x)}(ds) \mathbf{1}_{\sigma_x\le t} - \int_0^t \sum_{x\in I} \delta_{(\kappa_x,u-\tau_x)}(ds) h_{S_u}(s) \mathbf{1}_{\tau_x\le u<\sigma_x} du \bigg),
\end{align*}
then $(f,M_t) = M^f_t$.
For $g\in C^0(\S)$ and $\varphi\in C^0(\T)$, $\int_0^t \varphi(u) d(g,M_u)$ is a martingale with
\begin{multline*}
\bigg<\int_0^\cdot \varphi(u) d(g,M_u) \bigg>_t = \int_0^t \varphi^2(u) \bigg( \sum_{i_1\in\K} \sum_{i_2\in\K} g(i_1,0)g(i_2,0) w^{i_1i_2}_{S_u} \\
+ h_{S_u} g^2(\cdot,u) - 2 \sum_{i\in\K} g(i,0,u) h_{S_u} \widehat m^{i}_{S_u} g(\cdot,u), S_u \bigg) du.
\end{multline*}
Write $\int_0^t (\varphi(u)g,dM_u)$ for $\int_0^t \varphi(u) d(g,M_u)$.
Then, by the Monotone Class Theorem (see e.g. \cite[I.22.1]{DelMey78}),
$\int_0^t (f(\cdot,u),dM_u)$ is a martingale for any $f\in C^0(\S\times\T)$
whose predictable quadratic variation coincides with \eqref{E:QVMf2v}.

\subsection{Proof of the Law of Large Numbers} \label{S:LLNProof}

This is done by checking the tightness of the scaled sequence $\bar S^K$ and the uniqueness of the limiting process.
By Jakubowski's theorem \cite{Jak86},  $\{\bar S^K\}$ is tight in $\D(\T,\mathcal{M})$ if:
\begin{enumerate}
\item [(J1)]
For each $\eta>0$, there exists a compact set $\mathcal{C}_\eta \in \mathcal{M}$ such that
$$\liminf_{K\to\infty} \P\big(\bar S^K_t \in \mathcal{C}_\eta \, \forall t\in \T \big) > 1-\eta.$$
\item [(J2)]
For each $f\in C^1$, $\{(f,\bar S^K)\}$ is tight in $\D(\T,\R)$.
\end{enumerate}
In terms of the semimartingale decomposition $(f,\bar S^K_t) = V^{f,K}_t + \bar M^{f,K}_t$, 
(J2) reduces to the following Aldous-Rebolledo criteria:
\begin{enumerate}
\item [(J2a)]
For each $t\in\T$, $(f,\bar S^K_t)$ is tight,
that is, for each $\epsilon>0$, there exist $\delta>0$ such that for all $K$,
$\P\big( |(f,\bar S^K_t)| >\delta \big) <\epsilon.$
\item [(J2b)]
For each $\epsilon_1, \epsilon_2 >0$, there exist $\delta>0$ and $K_0 \ge1$ such that
for every sequence of stopping times $\tau^K \le T$,
\begin{gather}
\sup_{K>K_0} \sup_{\zeta<\delta} \mathbb{P} \big( | V_{\tau^K+\zeta\wedge T}^{f,K} - V_{\tau^K}^{f,K} | >\epsilon_1 \big) <\epsilon_2, \tag{i} \\
\sup_{K>K_0} \sup_{\zeta<\delta} \mathbb{P} \Big( \big| \big<\bar{M}^{f,K}\big>_{\tau^K+\zeta\wedge T} - \big<\bar{M}^{f,K}\big>_{\tau^K} \big| >\epsilon_1\Big) <\epsilon_2. \tag{ii}
\end{gather}
\end{enumerate}

\subsubsection{Preliminary estimates}

Recall the operator $L^K_S$ defined in \eqref{E:LK}.
Let $\hat L^K_S$ be such that
\begin{equation} \label{E:LhatK}
\widehat L^K_Sf = - h^K_Sf + \sum_{i\in\K} f(i,0) n_S^{i,K},
\end{equation}
so that $L^K_Sf = f' + \widehat L^K_Sf$.
Define also the operator $\Pi^K_S$ such that
$$\Pi^K_S f = \sum_{i_1\in\K} \sum_{i_2\in\K} f(i_1,0)f(i_2,0) w^{i_1i_2,K}_{S} + h^K_{S} f^2 \\
- 2 \sum_{i\in\K} f(i,0) h^K_{S} \widehat m^{i,K}_{S} f,$$
so that
$\big<M^{f,K}\big>_t = \int_0^t ( \Pi^K_{\bar S^K_u}f, S^K_u ) du$.

\begin{proposition} \label{P:PiLBnd}
Suppose (C0) holds.
Then, for any $f\in C^0$,
$$|\Pi^K_{S} f| \le c ||f||^2_{C^0} \q \textnormal{and} \q |\widehat L^K_{S} f| \le c ||f||_{C^0},$$
and for any $f\in C^1$,
$$|L^K_{S} f| \le c ||f||_{C^1}.$$
In particular,
$|\Pi^K_{S} 1| \le c$ and $|L^K_{S} 1| = |\widehat L^K_{S} 1| \le c$.
\end{proposition}

\begin{proof}
This follows immediately from the boundedness of the parameters and the definition of $||f||_{C^j}$.
\end{proof}

\begin{proposition} %\label{P:ESbar}
Suppose (C0) and (C3) hold.
Then,
\begin{equation} \label{E:1SK}
\E[(1,\bar S^K_t)] \le (1,\bar S^K_0) e^{ct}
\end{equation}
and
\begin{equation} \label{E:sup1SK}
\sup_{K\ge1} \E\Big[\sup_{t\le T}(1,\bar S^K_t)\Big] < \infty.
\end{equation}
\end{proposition}

\begin{proof}
From \eqref{E:SbarK} and Proposition \ref{P:PiLBnd},
$$(1,\bar S^K_t) \le (1,\bar S^K_0) + c_1 \int_0^t (1,\bar S^K_u) du + \frac1K M^{1,K}_t.$$
Taking expectation and applying Gronwall's inequality, we establish the first statement,
$\E[(1,\bar S^K_t)] \le (1,\bar S^K_0)e^{c_1t}$.

For the second statement, note that
$$\sup_{t\le T}(1,\bar S^K_t) \le (1,\bar S^K_0) + c_1 \int_0^T (1,\bar S^K_u) du + \frac1K \sup_{t\le T} M^{1,K}_t.$$
Taking expectation and using the first statement,
$$\E\Big[\sup_{t\le T}(1,\bar S^K_t)\Big] \le (1,\bar S^K_0) + c_1 (1,\bar S^K_0) e^{c_1T}T + \frac1K \E\Big[\sup_{t\le T} M^{1,K}_t\Big].$$
Now,
$$\E\Big[\sup_{t\le T} M^{1,K}_t\Big]^2 \le \E\Big[\sup_{t\le T} \big(M^{1,K}_t\big)^2 \Big] \le 4\E\big[ \big<M^{1,K}\big>_T \big]$$
by Doob's inequality; using Proposition \ref{P:PiLBnd} and the first statement again,
$$\E\big[ \big<M^{1,K}\big>_T \big] \le c_2 \int_0^T \E[(1,S^K_u)] du \le c_2 K (1,\bar S^K_0) e^{c_1T}T.$$
Therefore,
$$\E\Big[\sup_{t\le T}(1,\bar S^K_t)\Big] \le (1,\bar S^K_0) (1+ c_1 e^{c_1T}T) + \frac1{\sqrt{K}} c_3 (1,\bar S^K_0)^{1/2} e^{c_4T}T^{1/2},$$
which is bounded in $K$ by (C3).
\end{proof}

\subsubsection{Tightness of $\bar S^K$}

\begin{proposition} \label{P:SbarKTight}
Suppose (C0) and (C3) hold.
Then, the sequence $\bar S^K$ is tight in $\D(\T,\mathcal{M})$.
\end{proposition}

\begin{proof}
First, we show that (J1) holds.
From Markov's inequality and \eqref{E:sup1SK} we have
$$\P\big(\sup_{t\le T}(1,\bar S^K_t) >\delta\big) \le \frac{c}{\delta}.$$
This ensures the existence of $\delta_{\eta}$ such that
$\P(\sup_{t\le T}(1,\bar S^K_t)>\delta_{\eta}) \le \eta$.
Let
$$\mathcal{C}(\delta) = \{\mu\in \mathcal{M}: (1,\mu)\le \delta\}.$$
Then $\mathcal{C}(\delta)$ is compact and for any $\eta>0$,
$$\liminf_{K\to\infty} \P\big(\bar S^K_t \in \mathcal{C}(\delta_\eta) \, \forall t\in \T\big) >1-\eta.$$

Next, we show that (J2a) and (J2b) hold.
Condition (J2a) is immediate from Markov's inequality
$$\P\big( |(f,\bar S^K_t)| >\delta \big) \le \frac1\delta \E\big[ |(f,\bar S^K_t)| \big] \le \frac1\delta ||f||_{C^1} \E\big[ (1,\bar S^K_t) \big]$$
and by using \eqref{E:1SK}.

For (J2b)(i), note that, using Proposition \ref{P:PiLBnd},
\begin{align*}
| V_{\tau^K+\zeta\wedge T}^{f,K} - V_{\tau^K}^{f,K} |
%&\le \int_{\tau^K}^{\tau^K+\zeta\wedge T} (|L^K_{S^K_u}f|, \bar S^K_u) du \\
&\le c_1 ||f||_{C^1} \int_{\tau^K}^{\tau^K+\zeta\wedge T} (1, \bar S^K_u) du \\
&= c_1 ||f||_{C^1} \int_0^\zeta (1, \bar S^K_{\tau^K+u\wedge T}) du
\le c_1 ||f||_{C^1} \delta \sup_{t\le T} (1, \bar S^K_t)
\end{align*}
and thus
$$\E\big[ | V_{\tau^K+\zeta\wedge T}^{f,K} - V_{\tau^K}^{f,K} | \big] \le c_1 ||f||_{C^1} \delta \E\Big[ \sup_{t\le T} (1, \bar S^K_t) \Big]$$
is bounded by \eqref{E:sup1SK}.
Therefore, with Markov's inequality, we can choose $\delta$ such that (i) holds.

For (J2b)(ii), by Proposition \ref{P:PiLBnd},
\begin{multline*}
\big| \big<\bar{M}^{f,K}\big>_{\tau^K+\zeta\wedge T} - \big<\bar{M}^{f,K}\big>_{\tau^K} \big|
\le \frac1{K^2} \int_{\tau^K}^{\tau^K+\zeta\wedge T} ( |\Pi^K_{\bar S^K_u}f|, S^K_u ) du \\
\le c_2 ||f||_{C^1}^2 \frac1{K^2} \int_{\tau^K}^{\tau^K+\zeta\wedge T} ( 1, S^K_u ) du
\le c_2 ||f||_{C^1}^2 \delta \frac1{K^2} \sup_{t\le T} (1, S^K_t)
\end{multline*}
using a similar argument as for (i).
Thus,
$$\E\Big[ \big| \big<\bar{M}^{f,K}\big>_{\tau^K+\zeta\wedge T} - \big<\bar{M}^{f,K}\big>_{\tau^K} \big| \Big]
\le c_2 ||f||_{C^1}^2 \delta \frac1{K} \E\Big[ \sup_{t\le T} (1,\bar S^K_t) \Big]. $$
The proof is then complete by \eqref{E:sup1SK} and Markov's inequality.
\end{proof}

\begin{remark} \label{R:MKto0}
From the proof of Proposition \ref{P:SbarKTight}, we can see that the martingale $\bar M^{f,K} = \frac1K M^{f,K}$ is tight for any $f\in C^1$.
Moreover, it converges to 0 as $K$ tends to infinity, since the predictable quadratic variation vanishes.
\end{remark}

\subsubsection{Convergence of $\bar S^K$ and the limiting process}

Tightness implies the existence of a subsequence that converges.
We now identify the limit of $\bar S^K$ and show the uniqueness of the limiting process.

\begin{proposition}
Suppose (C0) -- (C3) hold.
Every limit point $\mathscr{S}$ of the sequence $\bar S^K$ satisfies the equation
\begin{equation} \label{E:fSbarLimit}
(f, \mathscr{S}_t) = (f, \mathscr{S}_0) + \int_0^t (L^\infty_{\mathscr{S}_u}f , \mathscr{S}_u) du
\end{equation}
for any $f\in C^1$ and $t\in\T$, 
where $L^\infty_Sf$ is as defined in \eqref{E:Linfty}.
\end{proposition}

\begin{proof}
From Remark \ref{R:MKto0}, the martingale sequence $\bar M^{f,K}$ vanishes as $K$ tends to infinity.
This together with the convergence of $\bar S^K_0$ by (C3), it remains to show the convergence of
$\int_0^t (L^K_{\bar S^K_u} f, \bar S^K_u) du$ to $\int_0^t (L^\infty_{\mathscr{S}_u} f, \mathscr{S}_u) du$.
Note that, by (C1) and (C2),
$$||h^K_{\bar S^K_u} - h^\infty_{\mathscr{S}_u}||_\infty
\le ||h^K_{\bar S^K_u} - h^K_{\mathscr{S}_u}||_\infty + ||h^K_{\mathscr{S}_u} - h^\infty_{\mathscr{S}_u}||_\infty
\to 0;$$
similarly for other model parameters.
Thus, $|| L^K_{\bar S^K_u} f - L^\infty_{\mathscr{S}_u} f ||_\infty \to 0$
and
$$\int_0^t \big| \big( L^K_{\bar S^K_u} f - L^\infty_{\mathscr{S}_u} f , \bar S^K_u \big)\big| du
\le \int_0^t \big|\big| L^K_{\bar S^K_u} f - L^\infty_{\mathscr{S}_u} f \big|\big|_\infty (1,\bar S^K_u) du
\to 0$$
by dominated convergence theorem.
Note also that, for any $f\in C^1$,
$$|L^\infty_S f| \le ||f'||_\infty + c_1||f||_\infty \le c_2||f||_{C^1}.$$
Thus,
$$\int_0^t \big| \big(L^\infty_{\mathscr{S}_u} f, \bar S^K_u-\mathscr{S}_u\big) \big| du
\le \int_0^t ||L^\infty_{\mathscr{S}_u} f||_\infty ||\bar S^K_u-\mathscr{S}_u|| du 
\le c_2 ||f||_{C^1} \int_0^t ||\bar S^K_u-\mathscr{S}_u|| du$$
vanishes as $K\to\infty$.
Hence,
\begin{multline*}
\bigg| \int_0^t \big( L^K_{\bar S^K_u} f, \bar S^K_u \big) du - \int_0^t \big( L^\infty_{\mathscr{S}_u}f, \mathscr{S}_u\big) du \bigg| \\
\le \int_0^t \bigg( \big| \big(L^K_{\bar S^K_u} f - L^\infty_{\mathscr{S}_u} f, \bar S^K_u\big) \big| + \big| \big(L^\infty_{\mathscr{S}_u} f, \bar S^K_u - \mathscr{S}_u\big) \big| \bigg) du
\end{multline*}
converges to zero.
\end{proof}

It remains to show the uniqueness of the solution to \eqref{E:fSbarLimit}.
To do this, we introduce an alternative representation to \eqref{E:fSbarLimit}.

\begin{proposition} \label{P:AltRepS}
For $\phi \in C^1$, define the shift operator $\widehat\varTheta_r$ such that
for each $r\in \T$ and $s=(i,v)$, $v\in[0,\omega-r]$, 
$\widehat\varTheta_r\phi(s) = \widehat\varTheta_r\phi(i,v) = \phi(i,v+r)$ 
and $\widehat\varTheta_r\phi \in C^0$ with $||\widehat\varTheta_r\phi||_{C^0} \le c||\phi||_{C^0}$.
(This is possible by reflecting $\phi$ about $v=\omega-r$, 
in which case $||\widehat\varTheta_r\phi||_{C^0} \le ||\phi||_{C^0}$.)
Then \eqref{E:fSbarLimit} is equivalent to the following:
for $\phi\in C^1$,
\begin{equation} \label{E:phiSbarLimit}
(\phi,\mathscr{S}_t) = (\widehat\varTheta_t\phi, \mathscr{S}_0) + \int_0^t \bigg( -h^\infty_{\mathscr{S}_u} \widehat\varTheta_{t-u}\phi + \sum_{i\in\K} \widehat\varTheta_{t-u}\phi(i,0) n^{i,\infty}_{\mathscr{S}_u} , \mathscr{S}_u \bigg) du.
\end{equation}
\end{proposition}

\begin{proof}
Let $g(s,t) = f(s)\varphi(t)$,
where $f$ and $\varphi$ are functions in $C^1(\T)$ and $C^1(\S)$ respectively.
Then,
\begin{align*}
(g(\cdot,t),\mathscr{S}_t) %&= \varphi(t) (f,\bar S_t) \\
%&= \varphi(0) (f,\bar S_0) + \int_0^t (f,\bar S_u) d\varphi(u) + \int_0^t \varphi(u) d(f,\bar S_u) \\
&= \varphi(0) (f,\mathscr{S}_0) + \int_0^t \varphi'(u) (f,\mathscr{S}_u) du + \int_0^t \varphi(u) (L^\infty_{\mathscr{S}_u} f, \mathscr{S}_u) du \\
&= (g(\cdot,0),\mathscr{S}_0) + \int_0^t \Big( \partial_1g(\cdot,u) + \partial_2g(\cdot,u) - h^\infty_{\mathscr{S}_u}g(\cdot,u) 
+ \sum_{i\in\K} g(i,0,u) n^{i,\infty}_{\mathscr{S}_u} , \mathscr{S}_u \Big) du. 
\end{align*}
By the Monotone Class Theorem (e.g. \cite[I.22.1]{DelMey78}), 
this holds for any $g \in C^{1,1}(\S\times\T)$.
Now, fix $t\in\T$ and $\phi\in C^1(\S)$, take $g(s,u) = \widehat\varTheta_{t-u}\phi(s) = \phi(i,v+t-u)$ for $s\in\S$ and $u\in[0,t]$, then we obtain \eqref{E:phiSbarLimit}.
From \eqref{E:phiSbarLimit}, we can also recover \eqref{E:fSbarLimit} by noting that
$\widehat\varTheta_{t-u}\phi(s) = \phi(s) + \int_u^t \widehat\varTheta_{w-u} \phi'(s) dw$
and applying Fubini's theorem.
\end{proof}

Now, we can show the uniqueness of the solution to \eqref{E:fSbarLimit}
by showing the uniqueness of the solution to \eqref{E:phiSbarLimit}.

\begin{proposition}
If $\mathscr{S}^1$ and $\mathscr{S}^2$ both satisfy \eqref{E:phiSbarLimit} with $\mathscr{S}^1_0=\mathscr{S}^2_0$, then $\mathscr{S}^1=\mathscr{S}^2$.
\end{proposition}

\begin{proof}
Let
$$\widehat L^\infty_Sf = -h^\infty_Sf + \sum_{i\in\K}f(i,0)n^{i,\infty}_S$$
so that \eqref{E:phiSbarLimit} becomes 
$(\phi,\mathscr{S}_t) = (\widehat\varTheta_t\phi, \mathscr{S}_0) + \int_0^t \big( \widehat L^\infty_{\mathscr{S}_u} \widehat\varTheta_{t-u}\phi, \mathscr{S}_u \big) du$.
We have, for $\phi\in C^1$,
\begin{align*}
|(\phi,\mathscr{S}^1_t-\mathscr{S}^2_t)| 
&\le \int_0^t \big|\big( \widehat L^\infty_{\mathscr{S}^1_u} \widehat\varTheta_{t-u} \phi - \widehat L^\infty_{\mathscr{S}^2_u} \widehat\varTheta_{t-u} \phi, \mathscr{S}^1_u\big)\big| + \big|\big( \widehat L^\infty_{\mathscr{S}^2_u} \widehat\varTheta_{t-u} \phi, \mathscr{S}^1_u-\mathscr{S}^2_u \big)\big| du \\
&\le \int_0^t || \widehat L^\infty_{\mathscr{S}^1_u} \widehat\varTheta_{t-u} \phi - \widehat L^\infty_{\mathscr{S}^2_u} \widehat\varTheta_{t-u} \phi||_\infty (1, \mathscr{S}^1_u) + || \widehat L^\infty_{\mathscr{S}^2_u} \widehat\varTheta_{t-u} \phi||_\infty ||\mathscr{S}^1_u-\mathscr{S}^2_u|| du.
\end{align*}
Now, by (C1),
\begin{align*}
||h^\infty_{\mu} - h^\infty_{\nu}||_\infty
&\le ||h^\infty_{\mu} - h^K_{\mu}||_\infty + ||h^K_{\mu} - h^K_{\nu}||_\infty + ||h^K_{\nu} - h^\infty_{\nu}||_\infty \\
&\le ||h^\infty_{\mu} - h^K_{\mu}||_\infty + c||\mu-\nu|| + ||h^K_{\nu} - h^\infty_{\nu}||_\infty;
\end{align*}
taking the limit $K\to\infty$ on both sides, we obtain
$||h^\infty_{\mu} - h^\infty_{\nu}||_\infty \le c_1||\mu-\nu||$.
Similarly for other model parameters.
Thus, $||\widehat L^\infty_{\mu}f - \widehat L^\infty_{\nu}f||_\infty \le c_2 ||f||_\infty ||\mu-\nu||$.
We also have by (C0) and (C2) that
$||\widehat L^\infty_{\mu}f||_\infty \le c_3||f||_\infty$.
Finally, since also $||\widehat\varTheta_r\phi||_{C^0} \le c ||\phi||_{C^0}$ for any $r\in\T$,
we have
$$|(\phi,\mathscr{S}^1_t-\mathscr{S}^2_t)| \le \int_0^t \big( c_4 ||\phi||_{C^0} ||\mathscr{S}^1_u-\mathscr{S}^2_u|| (1,\mathscr{S}^1_u) + c_5 ||\phi||_{C^0} ||\mathscr{S}^1_u-\mathscr{S}^2_u|| \big) du.$$
Using (C0) and Gronwall's inequality, we can show that
$(1,\mathscr{S}^1_u) \le (1,\mathscr{S}^1_0)e^{c_6u}$.
Therefore, it follows that
\begin{equation} \label{E:LLNUniq}
|(\phi,\mathscr{S}^1_t-\mathscr{S}^2_t)| \le c_7 ||\phi||_{C^0} \big( 1+ (1,\mathscr{S}^1_0)e^{c_6T} \big) \int_0^t  ||\mathscr{S}^1_u-\mathscr{S}^2_u|| du.
\end{equation}
Note that for any $\phi\in C^0$, there exists a sequence $\{\phi_n\}$ such that $\phi_n\in C^1$ and $\phi_n$ converges uniformly to $\phi$.
Thus, \eqref{E:LLNUniq} holds for any $\phi\in C^0$,
and $||\mathscr{S}^1_t-\mathscr{S}^2_t|| =0$ by Gronwall's inequality.
\end{proof}

Putting all together, we established the convergence of $\bar S^K$ as stated in Theorem \ref{T:LLN}.

%%%%%

\subsection{Proof of the Central Limit Theorem} \label{S:CLTProof}

This is in line with the single-type case \cite{FanEtal19} and 
starts from an alternative representation of $Z^K_t$, which is a similar trick as in Proposition \ref{P:AltRepS}.

\subsubsection{Alternative representation} \label{S:AltRep}

\begin{remark} \label{R:ShiftOp}
Like in \cite{FanEtal19}, we take $\varTheta_r\phi$ as the function on $\S$ such that
for each $r\in\T$ and $s=(i,v)$, $v\in[0,\omega-r]$, 
$\varTheta_r\phi(s) = \varTheta_r\phi(i,v) = \phi(i,v+r)$ 
and, whenever $\phi\in W^j$, then $\varTheta_r\phi \in W^j$ with
\begin{equation} \label{E:ThetaPhiNorm}
||\varTheta_r\phi||_{W^j} \le c ||\phi||_{W^j}.
\end{equation}
The existence of such a function was proved in \cite{FanEtal19}.
\end{remark}

Equation \eqref{E:ZKL} can also be extended to
test functions dependent on $t$, that is, for $f\in C^{1,1}(\S\times\T)$.
Then, for fix $t$, by taking $f(s,u) = \varTheta_{t-u}\phi(s)$ for $u\le t$,
we have the following equation.

\begin{proposition} \label{P:ZAltRep}
For $\phi\in C^1$ and $t\in\T$,
\begin{multline} \label{E:phiZK}
(\phi, Z^K_t) = (\varTheta_t\phi, Z^K_0) \\
+ \sqrt{K} \int_0^t \Big( - (h_{\bar S_u^K}^K - h^\infty_{\bar{S}_u}) \varTheta_{t-u}\phi + \sum_{i\in\K} \varTheta_{t-u}\phi(i,0) (n_{\bar S_u^K}^{i,K} - n^{i,\infty}_{\bar{S}_u}) , \bar{S}_u \Big) du \\
+ \int_0^t \Big( - h_{\bar S_u^K}^K \varTheta_{t-u}\phi + \sum_{i\in\K} \varTheta_{t-u}\phi(i,0) n_{\bar S_u^K}^{i,K} , Z_u^K \Big) du
+ \int_0^t (\varTheta_{t-u}\phi, d\tilde M^K_u),
\end{multline}
where $\tilde M=\frac1{\sqrt{K}}M$ and $M$ is the measure as defined in Remark \ref{R:MeasureM}.
\end{proposition}

\begin{proof}
Let $f\in C^1(\S)$ and $\varphi\in C^1(\T)$.
For $g:\S\times\T\to\R$ such that $g(s,t) = \varphi(t)f(s)$,
\begin{align*}
&(g(\cdot,t),Z^K_t) %= \varphi(t)(f,Z^K_t)
%= \varphi(0)(f,Z^K_0) + \int_0^t \varphi'(u)(f,Z^K_u)du + \int_0^t \varphi(u)d(f,Z^K_u) \\
= (\varphi(0)f,Z^K_0) + \int_0^t \varphi'(u)(f,Z^K_u)du \\
&\q\q\q+ \sqrt{K} \int_0^t \varphi(u) \Big( - (h_{S_u^K}^K - h^\infty_{\bar{S}_u}) f + \sum_{i=1}^k f(i,0) (n_{S_u^K}^{i,K} - n^{i,\infty}_{\bar{S}_u}) , \bar{S}_u \Big) du \\
&\q\q\q+ \int_0^t \varphi(u) \Big( f' - h_{S_u^K}^K f + \sum_{i=1}^k f(i,0) n_{S_u^K}^{i,K} , Z_u^K \Big) du
+ \int_0^t \varphi(u) d(f,\tilde M^K_u).
\end{align*}
We shall write $\int_0^t (\varphi(u)f,d\tilde M^K_u)$ for $\int_0^t \varphi(u) d(f,\tilde M^K_u)$.
Then, we have
\begin{equation} \label{E:gZK}
\begin{split}
&(g(\cdot,t),Z^K_t) = (g(\cdot,0),Z^K_0) \\
&\q\q+ \sqrt{K} \int_0^t \Big( - (h_{S_u^K}^K - h^\infty_{\bar{S}_u}) g(\cdot,u) + \sum_{i=1}^k g((i,0),u) (n_{S_u^K}^{i,K} - n^{i,\infty}_{\bar{S}_u}) , \bar{S}_u \Big) du \\
&\q\q+ \int_0^t \Big( \partial_1g(\cdot,u) + \partial_2g(\cdot,u) - h_{S_u^K}^K g(\cdot,u) + \sum_{i=1}^k g((i,0),u) n_{S_u^K}^{i,K} , Z_u^K \Big) du \\
&\q\q+ \int_0^t (g(\cdot,u), d\tilde M^K_u).
\end{split}
\end{equation}
By the Monotone Class Theorem (e.g. \cite[I.22.1]{DelMey78}), \eqref{E:gZK} holds for $g\in C^{1,1}(\S\times\T)$.
Now, taking $g(s,u) = \varTheta_{t-u}\phi(s) = \varTheta_{t-u}\phi(i,v) := \phi(i,v+t-u)$,
we can conclude \eqref{E:phiZK}.
\end{proof}

\subsubsection{Preliminary estimates}

Recall the operator $\widehat L^K_S$ defined in \eqref{E:LhatK} and its limit $\widehat L^\infty_S$.

\begin{proposition} \label{P:ParaBound}
Suppose (A2) and (A3) hold.
Then, for any $f\in W^j$, $j\in\N$, and $t\in \T$,
\begin{equation} \label{E:KDiffLBnd}
%\big|\big|\sqrt{K}(L^K_{S^K_t}-L^\infty_{S_t}) f\big|\big|_\infty \equiv
\big|\big|\sqrt{K}(\widehat L^K_{\bar S^K_t}-\widehat L^\infty_{\bar S_t}) f\big|\big|_\infty
\le c (1+ ||Z^K_t||_{W^{-4}}) ||f||_{W^j}.
\end{equation}
\end{proposition}

\begin{proof}
By the triangle inequality,
\begin{align*}
%\big|h^K_{S^K_t} - h^\infty_{\bar S_t}\big|(s)
%&\le \big|h^K_{S^K_t} - h^\infty_{\bar S^K_t}\big|(s) + \big|h^\infty_{\bar S^K_t} - h^\infty_{\bar S_t} - \partial_S h^\infty_{\bar S_t}(\bar S^K_t-\bar S_t)\big|(s) \\
%&\q\q+ \big|\partial_S h^\infty_{\bar S_t} (\bar S^K_t-\bar S_t)\big|(s)\\
||h^K_{\bar S^K_t} - h^\infty_{\bar S_t}||_\infty
&\le ||h^K_{\bar S^K_t} - h^\infty_{\bar S^K_t}||_\infty + ||h^\infty_{\bar S^K_t} - h^\infty_{\bar S_t} - \partial_S h^\infty_{\bar S_t}(\bar S^K_t-\bar S_t)||_\infty \\
&\q\q+ ||\partial_S h^\infty_{\bar S_t}||_{\mathbb{L}^{-4}} ||\bar S^K_t-\bar S_t||_{W^{-4}}
\end{align*}
and thus,
\begin{multline*}
\sqrt{K}||h^K_{\bar S^K_t} - h^\infty_{\bar S_t}||_\infty
\le \sqrt{K} ||h^K_{\bar S^K_t} - h^\infty_{\bar S^K_t}||_\infty \\
+ \frac{||Z^K_t||_{W^{-4}}}{||\bar S^K_t-\bar S_t||_{W^{-4}}}||h^\infty_{\bar S^K_t} - h^\infty_{\bar S_t} - \partial_S h^\infty_{\bar S_t}(\bar S^K_t-\bar S_t)||_\infty
+ c_1 ||Z^K_t||_{W^{-4}},
\end{multline*}
where the bound in the last term is due to (A3).
The first term is bounded by (A2) and the second term is bounded by (A3).
Thus, $\sqrt{K} ||h^K_{\bar S^K_t} - h^\infty_{\bar S_t}||_\infty \le c (1+ ||Z^K_t||_{W^{-4}})$ for $t\in \T$.
Similarly, $\sqrt{K} ||n^{i,K}_{\bar S^K_t} - n^{i,\infty}_{\bar S_t}||_\infty \le c (1+ ||Z^K_t||_{W^{-4}})$ for any $i\in\K$.
From these two inequalities, we have \eqref{E:KDiffLBnd} immediately.
\end{proof}

Notice that the operator $L^K_S$ maps a function from $W^j$ to $W^{j-1}$ due to the derivative $f'$.
We shall write $\mathcal{L}^{j,j'} = L(W^j,W^{j'})$ for the space of linear operators from $W^j$ to $W^{j'}$.
Then, we have the following results.

\begin{proposition} \label{P:LKNormBnd}
Suppose (A2) holds. Then we have
\begin{gather}
\tag{i} \label{E:LhatKBnd} \sup_{K,S} ||\widehat L^K_S||_{\mathcal{L}^{j,j}} \le c, \q j\le 3; \\
\tag{ii} \label{E:LKBnd} \sup_{K,S} ||L^K_S||_{\mathcal{L}^{j,j-1}} \le c, \q 2\le j\le 4.
\end{gather}
\end{proposition}

\begin{proof}
First, note that if $f\in C^j$ and $g\in W^j$, then 
\begin{equation} \label{E:fgNorm}
||fg||_{W^j} \le c ||f||_{C^j} ||g||_{W^j}.
\end{equation}
The triangle inequality then yields 
$$||\widehat L^K_{\bar S^K_t} f||_{W^j}
%\le ||h^K_{\bar S^K_t}f||_{W^j} + \sum_{\i=1}^k |f(\i,0)| ||n^{\i,K}_{\bar S^K_t}||_{W^j} \\
\le ||h^K_{\bar S^K_t}||_{C^j}||f||_{W^j} + ||f||_\infty \sum_{i\in\K} ||n^{i,K}_{\bar S^K_t}||_{W^j}
\le c_1 ||f||_{W^{j}}$$
by (A2) and embedding.
Thus, $||\widehat L^K_{\bar S^K_t}||_{\mathcal{L}^{j,j}} \le c_1$ and \eqref{E:LhatKBnd} follows.
For \eqref{E:LKBnd},
$$||L^K_S f||_{W^{j-1}} = ||f'||_{W^{j-1}} + ||\widehat L^K_Sf||_{W^{j-1}}
\le ||f||_{W^{j}} + c_1 ||f||_{W^{j-1}}
\le c_2 ||f||_{W^{j}},$$
by \eqref{E:LhatKBnd} and again embedding.
\end{proof}

Now, define the operator $\Lambda^K_t$ as
$$\Lambda^K_t f = \sqrt{K}\big( (\widehat L^K_{\bar S^K_t}-\widehat L^\infty_{\bar S_t}) f, \bar S_t\big)
+ \big(L^K_{\bar S^K_t}f, Z^K_t\big).$$
Then, Propositions \ref{P:ParaBound} and \ref{P:LKNormBnd} imply the following.

\begin{corollary} \label{C:Lambdaf}
Suppose that (A2) and (A3) hold.
Let $2\le j\le 4$. For $t\in\T$,
$$||\Lambda^K_t||_{W^{-j}} \le c_1 \big(1+ (1, \bar S_0)e^{c_2t} \big) \big(1+ ||Z^K_t||_{W^{-(j-1)}} \big).$$
\end{corollary}

\begin{proof}
Since
\begin{align*}
|\Lambda^K_t f|
%&\le \big| \sqrt{K}\big( (L^K_{\bar S^K_t}-L^\infty_{\bar S_t}) f, \bar S_t\big) \big|
%+ \big| \big(L^K_{\bar S^K_t}f, Z^K_t\big) \big| \\
%&\le \big(| \sqrt{K}(L^K_{\bar S^K_t}-L^\infty_{\bar S_t}) f|, \bar S_t\big)
%+ \big|\big| L_{\bar S_t^K}^K f \big|\big|_{W^{j-1}} ||Z_t^K||_{W^{-(j-1)}} \\
&\le \big(| \sqrt{K}(\widehat L^K_{\bar S^K_t}-\widehat L^\infty_{\bar S_t}) f|, \bar S_t\big)
+ \big|\big| L_{\bar S_t^K}^K f\big|\big|_{W^{j-1}} ||Z_t^K||_{W^{-(j-1)}},
\end{align*}
by Propositions \ref{P:ParaBound} and \ref{P:LKNormBnd}\eqref{E:LKBnd}, we have that
$$|\Lambda^K_t f|
\le c_1 \big(1+||Z^K_t||_{W^{-4}}\big) ||f||_{W^{j}} (1, \bar S_t)
+ c_2 ||f||_{W^{j-1}} ||Z_t^K||_{W^{-(j-1)}}.$$
Using embedding and the bound
\begin{equation} \label{E:1SInfty}
(1, \bar S_t) \le (1, \bar S_0) e^{ct}
\end{equation}
(due to Gronwall's inequality), we have
$$|\Lambda^K_t f| \le c_3 \big(1+||Z^K_t||_{W^{-(j-1)}}\big) ||f||_{W^{j}} (1, \bar S_0)e^{c_4t}
+ c_2 ||f||_{W^{j}} ||Z_t^K||_{W^{-(j-1)}}.$$
Rearranging and noting that $x\le 1+x$ for any $x$ complete the proof.
\end{proof}

Now, define the operator $\Gamma^K_t$ as
\begin{multline} \label{E:GammaKt}
\Gamma^K_t f = (\Pi^K_{\bar S^K_t}f, \bar S^K_t) \\
= \bigg( \sum_{i_1\in\K} \sum_{i_2\in\K} f(i_1,0)f(i_2,0) w^{i_1i_2,K}_{\bar S^K_t} + h^K_{\bar S^K_t} f^2
- 2 \sum_{i\in\K} f(i,0) h^K_{\bar S^K_t} \widehat m^{i,K}_{\bar S^K_t} f , \bar S^K_t \bigg),
\end{multline}
%that is, $\Gamma^K_t f = (\Pi^K_{\bar S^K_t}f, \bar S^K_t)$.
and let $(p^j_l)_{l\ge1}$ be a complete orthonormal basis of $W^j$.

\begin{proposition}
Let $j\in\N$. Suppose (C0) holds. Then, for $t\in\T$,
\begin{equation} \label{E:SumGamPj}
\Big|\sum_{l\ge1} \Gamma^K_t p^j_l\Big| \le c(1, \bar S^K_t)
\end{equation}
and for $u\le t$,
\begin{equation} \label{E:SumGamThetaPj}
\Big|\sum_{l\ge1} \Gamma^K_u \varTheta_{t-u} p^j_l\Big| \le c(1, \bar S^K_u).
\end{equation}
\end{proposition}

\begin{proof}
First, note that
\begin{equation} \label{E:CONB}
\sup_{s_1\in\S}\sup_{s_2\in\S} \Big|\sum_{l\ge1} p^j_l(s_1)p^j_l(s_2)\Big| \le c.
\end{equation}
This can be seen by considering the operator $\mathcal{H}_s:f \mapsto f(s)$ on $W^j$.
The Riesz Representation Theorem and Parseval's identity yield
$||\mathcal{H}_s||^2_{W^{-j}} = \sum_{l\ge1} (p^j_l(s))^2$ on one hand,
and
$|\mathcal{H}_sf| = |f(s)| \le ||f||_{C^0} \le c_1||f||_{W^j}$ gives
$||\mathcal{H}_s||_{W^{-j}} \le c_1$ on the other hand.
Thus, $\sum_{l\ge1} (p^j_l(s))^2 \le c_2$ for all $s\in\S$.
Now, for any $s_1\in\S$ and $s_2\in\S$,
$$|\sum_{l\ge1} p^j_l(s_1)p^j_l(s_2)|^2
\le c_3||\sum_{l\ge1} p^j_l(s_1)p^j_l ||^2_{W^j}
= c_3 \sum_{l\ge1} (p^j_l(s_1))^2 \le c_4.$$
Thus, \eqref{E:CONB} follows.
By this and (C0), \eqref{E:SumGamPj} follows immediately from \eqref{E:GammaKt}.

For \eqref{E:SumGamThetaPj}, observe that by (C0),
\begin{multline*}
\Big|\sum_{l\ge1} \Gamma^K_u \varTheta_{t-u} p^j_l \Big|
\le c_5 \bigg( \sum_{i_1\in\K} \sum_{i_2\in\K} \Big|\sum_{l\ge1} \varTheta_{t-u}p^j_l(i_1,0) \varTheta_{t-u}p^j_l(i_2,0) \Big| \\
+ \Big|\sum_{l\ge1} (\varTheta_{t-u} p^j_l)^2 \Big|
+ \sum_{i\in\K} \Big|\sum_{l\ge1} \varTheta_{t-u}p^j_l(i,0) \varTheta_{t-u}p^j_l \Big| ,\bar S^K_u \bigg),
\end{multline*}
and
\begin{align*}
&\q\q\sup_{s_1\in\K\times[0,u+a^*]} \sup_{s_2\in\K\times[0,u+a^*]} \Big| \sum_{l\ge1} \varTheta_{t-u}p^j_l(s_1) \varTheta_{t-u}p^j_l(s_2) \Big| \\
&= \sup_{i_1\in\K,w_1\in[0,u+a^*]} \sup_{i_2\in\K,w_2\in[0,u+a^*]} \Big| \sum_{l\ge1} p^j_l(i_1,w_1+t-u) p^j_l(i_2,w_2+t-u) \Big| \\
&\le \sup_{s_1\in\K\times[t-u,t+a^*]} \sup_{s_2\in\K\times[t-u,t+a^*]} \Big| \sum_{l\ge1} p^j_l(s_1) p^j_l(s_2) \Big|,
\end{align*}
which is finite by \eqref{E:CONB}.
Hence, \eqref{E:SumGamThetaPj} follows.
\end{proof}

\subsubsection{Bounds on the norm of $Z^K$}

\begin{proposition} \label{P:ZW-2Bound}
$$\sup_{t\le T} \sup_{K\ge1} \E \big[ ||Z^K_t||_{W^{-2}} \big] <\infty.$$
\end{proposition}

\begin{proof}
For this, we use the alternative representation of $Z^K$ as given in Proposition \ref{P:ZAltRep}.
Note that, for $\phi\in W^2$,
\begin{multline*}
(\phi, Z^K_t) = (\varTheta_t\phi, Z^K_0)
+ \sqrt{K} \int_0^t \big( (\widehat L^K_{\bar S^K_u}-\widehat L^\infty_{\bar S_u}) \varTheta_{t-u}\phi, \bar S_u\big) du \\
+ \int_0^t \big(\widehat L^K_{\bar S^K_u} \varTheta_{t-u}\phi, Z^K_u\big) du
+ \int_0^t (\varTheta_{t-u}\phi, d\tilde M^K_u).
\end{multline*}
Thus, by Propositions \ref{P:ParaBound} and \ref{P:LKNormBnd}, with embedding,
\begin{align*}
|(\phi, Z^K_t)| &\le ||\varTheta_t\phi||_{W^2} ||Z^K_0||_{W^{-2}}
+ c_1 \int_0^t (1+||Z^K_u||_{W^{-2}}) ||\varTheta_{t-u}\phi||_{W^2} (1,\bar S_u) du \\
&\q\q+ c_2 \int_0^t ||\varTheta_{t-u}\phi||_{W^2} ||Z^K_u||_{W^{-2}} du
+ \Big|\int_0^t (\varTheta_{t-u}\phi, d\tilde M^K_u)\Big| \\
&\le c_3 ||\phi||_{W^2} \bigg\{ ||Z^K_0||_{W^{-2}}
+ (1,\bar S_0)e^{c_4t}t \\
&\q\q+ \big(1+(1,\bar S_0)e^{c_4t}\big) \int_0^t ||Z^K_u||_{W^{-2}} du
+ \Big|\Big| \int_0^t \varTheta^*_{t-u}d\tilde M^K_u \Big|\Big|_{W^{-2}}\bigg\},
\end{align*}
where the last inequality above is due to \eqref{E:1SInfty}, \eqref{E:ThetaPhiNorm}
and that we write $\int_0^t \varTheta^*_{t-u}d\tilde M^K_u$ for an operator such that
$\big(f,\int_0^t \varTheta^*_{t-u}d\tilde M^K_u\big) = \int_0^t (\varTheta_{t-u}f, d\tilde M^K_u)$, resulting in an expression for $||Z^K_t||_{W^{-2}}$.

Further,
$\E\big[ || \int_0^t \varTheta^*_{t-u}d\tilde M^K_u ||_{W^{-2}} \big]
\le \E\big[ || \int_0^t \varTheta^*_{t-u}d\tilde M^K_u ||^2_{W^{-2}} \big]^{1/2}$.
The Riesz Representation Theorem and Parseval's identity yield, for $r\le t$,
\begin{multline*}
\E\bigg[ \Big|\Big| \int_0^r \varTheta^*_{t-u}d\tilde M^K_u \Big|\Big|^2_{W^{-2}} \bigg]
= \E\bigg[ \sum_{l\ge1} \Big( \int_0^r (\varTheta_{t-u}p^2_l, d\tilde M^K_u) \Big)^2 \bigg] \\
= \sum_{l\ge1} \E\bigg[ \Big< \int_0^\cdot (\varTheta_{t-u}p^2_l, d\tilde M^K_u) \Big>_r \bigg]
= \sum_{l\ge1} \E\bigg[ \int_0^r \Gamma^K_u \varTheta_{t-u}p^2_l du \bigg].
\end{multline*}
It then follows from \eqref{E:SumGamThetaPj} and \eqref{E:1SK} that
this is bounded by $c_5 (1, \bar S_0^K) e^{c_6r}r$.
Taking $r=t$, we have $\E\big[|| \int_0^t \varTheta^*_{t-u}d\tilde M^K_u ||^2_{W^{-2}} \big]
\le c_5 (1, \bar S_0^K) e^{c_6t}t$.
Therefore,
\begin{multline*}
\E\big[||Z^K_t||_{W^{-2}}\big]
\le c_7 \bigg\{ ||Z^K_0||_{W^{-2}}
+ (1,\bar S_0)e^{c_4T}T^2 \\
+ \big(1+(1,\bar S_0)e^{c_4T}\big) \int_0^t \E\big[||Z^K_u||_{W^{-2}}\big] du
+ (1, \bar S_0^K)^{1/2} e^{c_8T}T^{1/2} \bigg\}
\end{multline*}
and the proof is complete by Gronwall's inequality, with (C3) and (A4).
\end{proof}

\begin{proposition} \label{P:ZW-3Bound}
$$\sup_{K\ge1} \mathbb{E} \Big[ \sup_{t\le T} ||Z^K_t||_{W^{-3}} \Big] < \infty.$$
\end{proposition}

\begin{proof}
This relies on the bound in Proposition \ref{P:ZW-2Bound}.
 From \eqref{E:ZKL}, we have for $f\in W^3$,
\begin{align*}
|(f,Z^K_t)| &\le ||f||_{W^3} ||Z^K_0||_{W^{-3}} + \sqrt{K} \int_0^t \big( |L^K_{\bar S^K_u}f - L^\infty_{\bar S_u}f |, \bar{S}_u \big) du \\
&\q\q\q+ \int_0^t ||L^K_{\bar S^K_u}||_{\mathcal{L}^{3,2}} ||f||_{W^3} ||Z^K_u||_{W^{-2}} du + |\tilde{M}_t^{f,K}| \\
&\le ||f||_{W^3} ||Z^K_0||_{W^{-3}} + c_1 ||f||_{W^3} (1, \bar{S}_0)e^{c_2t} \int_0^t (1+||Z^K_u||_{W^{-2}}) du \\
&\q\q\q+ c_3 ||f||_{W^3} \int_0^t ||Z^K_u||_{W^{-2}} du + ||f||_{W^3} ||\tilde{M}_t^K||_{W^{-3}}
\end{align*}
by Propositions \ref{P:ParaBound} and \ref{P:LKNormBnd}, \eqref{E:1SInfty} and embedding.
This gives
\begin{multline*}
\sup_{t\le T} ||Z^K_t||_{W^{-3}} \le ||Z^K_0||_{W^{-3}} \\
+ c_4 \big(1+ (1, \bar{S}_0) e^{c_2T}\big) \int_0^T (1+||Z^K_u||_{W^{-2}}) du
+ \sup_{t\le T} ||\tilde{M}_t^K||_{W^{-3}}.
\end{multline*}
By the Riesz Representation Theorem and Parseval's identity again, now along with Doob's inequality, we have
\begin{equation} \label{E:T1forM}
\begin{split}
\E\Big[ \sup_{t\le T} ||\tilde{M}_t^K||^2_{W^{-3}} \Big]
&= \E\bigg[ \sup_{t\le T} \sum_{l\ge1} \big(\tilde{M}_t^{p^3_l,K}\big)^2 \bigg]
\le \sum_{l\ge1} \E\Big[ \sup_{t\le T} \big(\tilde{M}_t^{p^3_l,K}\big)^2 \Big] \\
&\le 4 \sum_{l\ge1} \E\Big[ \big<\tilde{M}^{p^3_l,K}\big>_T \Big]
= 4 \sum_{l\ge1} \E\bigg[ \int_0^T \Gamma^K_u p^3_l du \bigg] \\
&\le c_5 (1,\bar S^K_0) e^{c_6T}T,
\end{split}
\end{equation}
where the last inequality follows from \eqref{E:SumGamPj} and \eqref{E:1SK}.
Therefore,
\begin{multline*}
\E\Big[ \sup_{t\le T} ||Z^K_t||_{W^{-3}} \Big]
\le c_7 \bigg\{ ||Z^K_0||_{W^{-3}} \\
+ \big(1+ (1, \bar{S}_0)e^{c_2T}\big) \int_0^T \big(1+\E\big[||Z^K_u||_{W^{-2}}\big]\big) du
+ (1,\bar S^K_0)^{1/2} e^{c_8T}T^{1/2} \bigg\}
\end{multline*}
and (C3), (A4), and Proposition \ref{P:ZW-2Bound} complete the proof.
\end{proof}

\subsubsection{Tightness of $Z^K$} \label{S:AppTightZ}

$Z^K$ is tight in $\D(\T,W^{-4})$ if the following conditions are satisfied: 
\begin{itemize}
\item [(T1)]
For every $t\in\T$, $(Z^K_t)_{K\ge1}$ is tight in $W^{-4}$, that is, 
for every $\epsilon>0$, there exists a compact set $\mathcal C_\epsilon$ such that 
$\P(Z^K_t \in \mathcal C_\epsilon) > 1-\epsilon$ for all $K\ge1$, equivalently, 
$\P(Z^K_t \notin \mathcal C_\epsilon) \le \epsilon$ for all $K\ge1$. 
\item [(T2)]
Suppose $Z^K_t = \tilde V^K_t + \tilde M^K_t$. 
For each $\epsilon_1, \epsilon_2 >0$, there exists $\delta>0$ and $K_0\ge1$ such that for every sequence of stopping times $\tau^K \le T$,
\begin{itemize}
\item [(T2a)]
$\sup_{K>K_0} \sup_{\zeta<\delta} \P( ||\tilde V^K_{\tau^K+\zeta\wedge T} - \tilde V^K_{\tau^K}||_{W^{-4}} > \epsilon_1) < \epsilon_2$,
\item [(T2b)]
$\sup_{K>K_0} \sup_{\zeta<\delta} \P\big( \big| \big<\big<\tilde M^K\big>\big>_{\tau^K+\zeta\wedge T} - \big<\big<\tilde M^K\big>\big>_{\tau^K}\big| > \epsilon_1\big) < \epsilon_2$, 
\end{itemize}
where $\big<\big<\tilde M^K\big>\big>$ is defined such that 
$\big( ||\tilde M^K_t||_{W^{-4}}^2 - \big<\big<\tilde M^K\big>\big>_t \big)_{t\ge0}$ is a martingale.
\end{itemize}

\begin{proposition} \label{P:TightZM}
Both the sequences $Z^K$ and $\tilde M^K$ are tight in $\D(\T,W^{-4})$.
\end{proposition}

\begin{proof}
(T1) follows from Proposition \ref{P:ZW-2Bound}.
Let $B_{W^{-2}}(R) = \{\mu\in W^{-2}: ||\mu||_{W^{-2}} \le R\}$. 
Since $W^{-2}$ is Hilbert-Schmidt embedded in $W^{-4}$, $B_{W^{-2}}(R)$ is compact in $W^{-4}$. 
From Proposition \ref{P:ZW-2Bound}, for any $\epsilon>0$, there exists $R$ such that 
$\P(Z^K_t \notin B_{W^{-2}}(R)) \le \epsilon$ for all $K\ge1$, 
since $\P(Z^K_t \notin B_{W^{-2}}(R)) = \P(||Z^K_t||_{W^{-2}} > R) \le \frac1R \E[ ||Z^K_t||_{W^{-2}} ]$. 

For (T2a) and (T2b), we check the stronger conditions established in \cite[Theorem 11]{FanEtal19}: 
There exists a $K_0 \ge1$ such that 
\begin{itemize}
\item [(T2a')]
$\sup_{K\ge K_0} \E \big[ \sup_{t\le T} ||\Lambda^K_t||_{W^{-j}} \big] \le c_T$,
\item [(T2b')]
$\sup_{K\ge K_0} \E \big[ \sup_{t\le T} |\sum_{l\ge1} \Gamma^K_t p^j_l | \big] \le c_T$,
\end{itemize}
where $(p^j_l)_{l\ge1}$ is a complete orthonormal basis of $W^{j}$.

From Corollary \ref{C:Lambdaf}, we obtain
$$\E\Big[ \sup_{t\le T} ||\Lambda^K_t||_{W^{-4}} \Big] \le c_1 \big(1+ (1, \bar S_0)e^{c_2T} \big) \bigg(1+ \E\Big[\sup_{t\le T} ||Z^K_t||_{W^{-3}}\Big] \bigg)$$
and (T2a') holds as a result of Proposition \ref{P:ZW-3Bound}.
Finally, (T2b') is a result of \eqref{E:SumGamPj} and \eqref{E:sup1SK}.

The tightness of $\tilde M^K$ follows from \eqref{E:T1forM} and (T2b').
\end{proof}

We can further show that $Z^K$ and $\tilde M^K$ are C-tight,
that is, the two sequences are tight and all limit points of the sequences are continuous.

\begin{proposition} %\label{P:ZKCtight}
Both the sequences $Z^K$ and $\tilde M^K$ are C-tight, all limit points are elements of $\mathbb{C}(\T,W^{-4})$.
\end{proposition}

\begin{proof}
As in \cite{FanEtal19}, for C-tightness of $Z^K$, we show that (see e.g. \cite[Proposition VI 3.26(iii)]{JacShi03}),
%(or \cite[Chapter 3 Theorem 10.2]{EthKur05}),
for all $u\in\mathbb{T}$ and $\epsilon>0$,
$\lim_{K\to\infty}\mathbb{P}\big(\sup_{t\le u} ||\Delta Z^K_t||_{W^{-4}} >\epsilon\big) =0$.
Note that $Z^K$ jumps when $S^K$ jumps, which occurs when there is a birth or a death. 
Thus, for $f\in W^{4}$,
\begin{align*} %\label{E:ZKJump}
|(f, \Delta Z^K_t)| &= \frac{1}{\sqrt{K}} \big|(f,S^K_t-S^K_{t-})\big| \\
&\le \frac{1}{\sqrt{K}} \bigg( \sum_{i\in\K} \sup_{s\in\S} \big|\widecheck\xi^{i,K}_{\bar S^K_t}(s) f(i,0)\big| + \sum_{i\in\K} \sup_{s\in\S} \big|\widehat\xi^{i,K}_{\bar S^K_t}(s) f(i,0) - f(s)\big| \bigg) \\
&\le \frac{c}{\sqrt{K}}||f||_{W^{4}} (1+\Xi)
\end{align*}
by (A1),
giving $||\Delta Z^K_t||_{W^{-4}} \le \frac{c}{\sqrt{K}}(1+\Xi)$. Hence,
$$\mathbb{P}\Big(\sup_{t\le u} ||\Delta Z^K_t||_{W^{-4}} >\epsilon\Big)
\le \frac1\epsilon \mathbb{E}\Big[ \sup_{t\le u} ||\Delta Z^K_t||_{W^{-4}} \Big]
\le \frac1\epsilon \frac{c}{\sqrt{K}} \big(1+ \mathbb{E}[\Xi]\big),$$
which converges to zero as $K$ tends to infinity.

Observing that $Z^K$ and $\tilde M^K$ have the same discontinuities,
i.e. $\Delta Z^K_t = \Delta \tilde M^K_t$, %from the proof of Proposition \ref{P:ZKCtight}
$\tilde M^K$ also satisfies the conditions of being C-tight.
\end{proof}

\subsubsection{Convergence of $\tilde M^K$ and $Z^K$} \label{SS:LimitNUniq}

We now give the final steps in establishing the convergence of $Z^K$.

\begin{proposition} \label{P:MInfty}
The sequence $\tilde M^K$ convergeces weakly to $\tilde M^\infty$ such that for any $f\in W^4$,
$\tilde M^{f,\infty}_t \equiv (f,\tilde M^\infty_t)$, $t\in\T$, is a continuous Gaussian martingale
with predictable quadratic variation
\begin{multline} \label{E:MinftyPQV}
\big<\tilde M^{f,\infty}\big>_t = \int_0^t \bigg(
\sum_{i_1\in\K} \sum_{i_2\in\K} f(i_1,0)f(i_2,0) w^{i_1i_2,\infty}_{\bar S_u} + h^\infty_{\bar S_u} f^2 \\
- 2 \sum_{i\in\K} f(i,0) h^\infty_{\bar S_u} \widehat m^{i,\infty}_{\bar S_u} f , \bar S_u \bigg) du,
\end{multline}
\end{proposition}

\begin{proof}
Same as in \cite{FanEtal19}, this can be achieved by showing that
$\tilde M^{f,K}$ converges to a continuous Gaussian martingale $\tilde M^{f,\infty}$
with predictable quadratic variation \eqref{E:MinftyPQV}.
In view of the tightness of $\tilde M^K$,  $\tilde M^K$ converges to $\tilde M^\infty$.
\end{proof}

\begin{proposition}
The limiting process $\mathscr{Z}$ of the sequence $Z^K$ satisfies \eqref{E:fZInfty} for any $f\in W^{4}$ and $t\in\T$.
\end{proposition}

\begin{proof}
Every limit point $\mathscr{Z}$ of the sequence $Z^K$ satisfies,
for $\phi \in W^{4}$ and $t\in\T$,
\begin{multline} \label{E:phiZInfty}
(\phi,\mathscr{Z}_t) = (\varTheta_t\phi,\mathscr{Z}_0) 
+ \int_0^t \Big(-\partial_S h^\infty_{\bar{S}_u}(\mathscr{Z}_u) \varTheta_{t-u}\phi + \sum_{i\in\K} \varTheta_{t-u}\phi(i,0)\partial_S n^{i,\infty}_{\bar{S}_u}(\mathscr{Z}_u), \bar{S}_u\Big) du \\
+ \int_0^t \Big(- h^\infty_{\bar{S}_u}\varTheta_{t-u}\phi + \sum_{i\in\K} \varTheta_{t-u}\phi(i,0)n^{i,\infty}_{\bar{S}_u}, \mathscr{Z}_u\Big) du
+ \int_0^t (\varTheta_{t-u}\phi, d\tilde{M}^{\infty}_u).
\end{multline}
This can be established by showing the convergence of each and every term of \eqref{E:phiZK}. 
Furthermore, if $\mathscr{Z}^1$ and $\mathscr{Z}^2$ both are solutions to \eqref{E:phiZInfty} with $\mathscr{Z}^1_0=\mathscr{Z}^2_0$, 
by showing that $||\mathscr{Z}^1_t-\mathscr{Z}^2_t||_{W^{-4}}=0$ for $t\in \T$, we have the uniqueness.
Lastly, we note that \eqref{E:phiZInfty} is equivalent to \eqref{E:fZInfty}, 
which is followed by Fubini's theorem and that 
$\varTheta_{t-u}\phi(s) = \phi(s) + \int_u^t \varTheta_{w-u}\phi'(s)dw$.
\end{proof}

\subsection{Proof of Proposition \ref{P:EZtMeasure}} \label{S:LastProof}

\begin{proof} %[Proof of Proposition \ref{P:EZtMeasure}]
Using representation \eqref{E:phiZInfty} and noting that $\E\big[\int_0^t (\varTheta_{t-u}\phi, d\tilde{M}^{\infty}_u)\big]=0$, with $\nu_t: f \mapsto \E[(f,Z_t)]$, we have for $\phi\in W^4$,
\begin{multline} \label{E:PhiNu}
(\phi,\nu_t) = (\varTheta_t\phi,\nu_0) 
+ \int_0^t \Big(-(g^h_{\bar{S}_u,\ast},\nu_u) \varTheta_{t-u}\phi + \sum_{i\in\K} \varTheta_{t-u}\phi(i,0) (g^{n,i}_{\bar{S}_u,\ast},\nu_u), \bar{S}_u\Big)_\ast du \\
+ \int_0^t \Big(- h^\infty_{\bar{S}_u}\varTheta_{t-u}\phi + \sum_{i\in\K} \varTheta_{t-u}\phi(i,0)n^{i,\infty}_{\bar{S}_u}, \nu_u\Big) du.
\end{multline}
Thus, under the assumptions in the statement and (A2),
$$||\nu_t||_{W^{-4}} \le ||\nu_0||_{W^{-4}} + c_1 \big(1+(1,\bar S_0)e^{c_2t}\big) \int_0^t ||\nu_u||_{W^{-4}} du.$$
Gronwall's inequality then gives
$$||\nu_t||_{W^{-4}} \le ||\nu_0||_{W^{-4}} e^{c_1 \big(1+(1,\bar S_0)e^{c_2T}\big)T}.$$

Now, let $(\phi_m)_m$ be a sequence of functions in $C^\infty$ converging to $\phi \in C^0$.
By dominated convergence theorem, each term in \eqref{E:PhiNu} with $\phi_m$ converges.
Thus, \eqref{E:PhiNu} holds for $\phi\in C^0$.
Moreover, $\nu_t$ is a bounded linear operator.
Therefore, $\nu_t \in C^{-0}$; in other words, $\nu_t$ defines a signed measure.
\end{proof}

\section{Appendices for Section \ref{S:Monogamy}} \label{S:Appendix-Monogamy}

\subsection{Semimartingale representation of the serial monogamy mating system} \label{S:Appendix-Monogamy-model}

To obtain an equation for the population structure in the serial monogamy mating system, we need to define a few more terms. 
Let $R(v,w,t)$ be the number of marriages by time $t$ with females at ages not greater than $v$ and males at ages not greater than $w$. 
Let $Q^1(v,w,t)$ (resp. $Q^2(v,w,t)$) be the number of cases by time $t$ where female (resp. male) partners died at age not greater than $v$ when their male (resp. female) partners were at age not greater than $w$, and let $Q^3(v,w,t)$ count the number of events by time $t$ where couples are separated while both of the mates are alive and with ages not greater than $v$ and $w$.
We also let $D^\F(v,t)$ (resp. $D^\M(v,t)$) be the number of single females (resp. males) who died by time $t$ at age not greater that $v$, and $B^\F(t)$ (resp. $B^\M(t)$) be the number of females (reps. males) born by time $t$. 
Then we have, for test functions $f$ as specified on Section \ref{S:Monogamy-model},
\begin{align*}
(f,S^\F_t) &= (f,S^\F_0) + \int_0^t (\partial_v f,S^\F_u) du + f(\F,0,\infty) B^\F([0,t]) 
- \int_{\A\times[0,t]} f(\F,v,\infty) D^\F(dv,du) \\ 
&\q\q+ \int_{\A\times\A\times[0,t]} f(\F,v,\infty) Q^2(dv,dw,du) 
+ \int_{\A\times\A\times[0,t]} f(\F,v,\infty) Q^3(dv,dw,du) \\
&\q\q- \int_{\A\times\A\times[0,t]} f(\F,v,\infty) R(dv,dw,du), \\
(f,S^\M_t) &= (f,S^\M_0) + \int_0^t (\partial_w f,S^\M_u) du + f(\M,\infty,0) B^\M([0,t]) 
- \int_{\A\times[0,t]} f(\M,\infty,w) D^\M(dv,du) \\
&\q\q+ \int_{\A\times\A\times[0,t]} f(\M,\infty,w) Q^1(dv,dw,du) 
+ \int_{\A\times\A\times[0,t]} f(\M,\infty,w) Q^3(dv,dw,du) \\
&\q\q- \int_{\A\times\A\times[0,t]} f(\M,\infty,w) R(dv,dw,du), \\
(f,S^\FM_t) &= (f,S^\FM_0) + \int_0^t (\partial_v f + \partial_w f,S^\FM_u) du 
- \int_{\A\times\A\times[0,t]} f(\FM,v,w) Q^3(dv,dw,du) \\
&\q\q- \int_{\A\times\A\times[0,t]} f(\FM,v,w) Q^1(dv,dw,du) 
- \int_{\A\times\A\times[0,t]} f(\FM,v,w) Q^2(dv,dw,du) \\
&\q\q+ \int_{\A\times\A\times[0,t]} f(\FM,v,w) R(dv,dw,du). 
\end{align*}
Compensating the birth, death and marriage terms, we obtain a semimartingale representation for each of the three types, which can be combined and written as \eqref{E:semimartB}.
For instance, the following processes are local martingales: 
\begin{gather*}
B^i([0,t]) - \int_0^t (bm^i, S^\FM_u) du - \int_0^t (bm^i, S^\F_u) du, \\
\int_{\A\times[0,t]} g(v) D^\i(dv,du) - \int_0^t (h^ig, S^i_u) du, \\
\int_{\A\times\A\times[0,t]} g(v,w) R(dv,dw,du) - \int_0^t ((\rho(\Cdot,\ast)g(\Cdot,\ast), S^\F_u)_{\Cdot}, S^\M_u)_\ast du, \\
\int_{\A\times\A\times[0,t]} g(v,w) Q^i(dv,dw,du) - \int_0^t (h^ig, S^\FM_u) du. 
\end{gather*}
The predictable quadratic variation of the martingale can also be obtained in the usual way as in Section \ref{S:AppA-model}.  

\subsection{Total population size in serial monogamy mating system} \label{S:Appendix-Monogamy-TPS}

Before we proceed to establish the tightness of $\{\bar S^K_t\}$, we first consider the total population size, by taking the test function $f(i,v,w) = \1_\F(i)+\1_\M(i)+2\1_\FM(i)$. 
This gives a simple dynamics equation that can be easily analysed and controlled, the results of which will be useful in establishing the tightness of $\{\bar S^K_t\}$.

With a slight abuse of notation, write $X_t = (\1_\F+\1_\M+2\1_\FM,S_t)$ and $\bar X^K_t = X^K_t/K$. 
Note that we have 
\begin{multline} \label{E:Xbar}
\bar X^K_t = \bar X^K_0 
- \int_0^t \big(h^{\F,K}_{\bar S^K_u} \mathbf{1}_\F + h^{\M,K}_{\bar S^K_u} \mathbf{1}_\M - (h^{\F,K}_{\bar S^K_u}+h^{\M,K}_{\bar S^K_u}) \1_\FM, \bar S^K_u\big) du \\
+ \int_0^t \Big( (m^{\F,K}_{\bar S^K_u} + m^{\M,K}_{\bar S^K_u}) b^K_{\bar S^K_u} (\mathbf{1}_\FM+\mathbf{1}_\F), \bar S^K_u\Big) du 
+ \bar M^{X,K}_t,
\end{multline}
with
\begin{multline*}
\left<\bar M^{X,K}\right>_t = \int_0^t \frac1K \big(h^{\F,K}_{\bar S^K_u} \mathbf{1}_\F + h^{\M,K}_{\bar S^K_u} \mathbf{1}_\M + (h^{\F,K}_{\bar S^K_u}+h^{\M,K}_{\bar S^K_u}) \1_\FM, \bar S^K_u\big) du \\
+ \int_0^t \frac1K \Big( \big[\gamma^{\F\F,K}_{\bar S^K_u} + \gamma^{\M\M,K}_{\bar S^K_u} + 2 \gamma^{\F\M,K}_{\bar S^K_u}\big] b^K_{\bar S^K_u} (\mathbf{1}_\FM+\mathbf{1}_\F), \bar S^K_u\Big) du. 
\end{multline*}

Since the reproduction parameters are bounded, by Gronwall's inequality, we have
$\E[\bar X^K_t] \le \bar X^K_0 e^{ct}$ and thus $\sup_K \E[\bar X^K_t] < \infty$. 
In fact, we can further show that 
\begin{equation} \label{E:EsupXBnd}
\sup_K \E\Big[\sup_{t\in\T} \bar X^K_t\Big] < \infty. 
\end{equation}
From \eqref{E:Xbar}, notice that $(1,\bar S^K_u) \le \bar X^K_u$, we have 
$$\E\Big[\sup_{t\in\T} \bar X^K_t\Big] \le \bar X^K_0 + c_1 \int_0^T \E\big[\bar X^K_u\big] du + \E\Big[ \sup_{t\in\T} \bar M^{X,K}_t \Big].$$
Apply Jensen's and Doob's inequalities, 
\begin{multline*}
\E\Big[ \sup_{t\in\T} \bar M^{X,K}_t \Big]^2 
\le \E\Big[ \sup_{t\in\T} \big(\bar M^{X,K}_t\big)^2 \Big] 
\le 4 \E\Big[\big<\bar M^{X,K}\big>_T\Big] \\
\le c_2 \frac1K \int_0^T \E\big[\bar X^K_u\big] du 
\le c_2 \frac1K \int_0^T \bar X^K_0 e^{cu} du 
\le c_2 \frac1K \bar X^K_0 e^{cT}T 
\end{multline*}
and thus 
$$\E\Big[\sup_{t\in\T} \bar X^K_t\Big] \le \bar X^K_0 + c_1 \int_0^T \bar X^K_0 e^{cu} du + \Big(c_2 \frac1K \bar X^K_0 e^{cT}T \Big)^{1/2},$$
which gives \eqref{E:EsupXBnd}. 

In a similar way, we can bound $\E[(\bar X^K_t)^2]$ and show that 
$$\sup_K \E\Big[ \sup_{t\in\T} (\bar X^K_t)^2 \Big] < \infty.$$

\subsection{Proof of the Law of Large Numbers (serial monogamy mating system)} \label{S:Appendix-Monogamy-LLN}

\begin{proposition}
The sequence $\{\bar S^K_t\}$ is tight in $\D(\T,\mathcal{M})$. 
\end{proposition}

\begin{proof}
The tightness is obtained by establishing (J1) and (J2). 
Observe that 
$$\P\Big( \sup_{t\in\T} (1,\bar S^K_t) > \delta \Big) \le \frac1\delta \E\Big[ \sup_{t\in\T} (1,\bar S^K_t) \Big] \le \frac1\delta \E\Big[ \sup_{t\in\T} X^K_t \Big] \le \frac c \delta.$$ 
Thus, for each $\eta$, there exists $\delta_\eta$ such that 
$\P( \sup_{t\in\T} (1,S^K_t) > \delta_\eta) \le \eta$. 
Therefore (J1) holds with compact set 
$$C(\delta_\eta) = \{ \mu\in\M: (1,\mu) \le \delta_\eta\}.$$

For (J2a), we have from the semimartingale representation of $(f,\bar S^K_t)$, 
$$\E\big[ (f,\bar S^K_t) \big] 
\le c_1 ||f||_{C^{1,1}} \Big\{ (1,\bar S^K_0) + \int_0^t \Big( \E\big[ (1,\bar S^K_u) \big] + \E\big[ (1,\bar S^K_u)^2 \big] \Big) du + \E\big[ |\bar M^{f,K}_t| \big] \Big\}$$
and
\begin{multline*}
\E\big[ |\bar M^{f,K}_t| \big]^2 \le \E\big[ |\bar M^{f,K}_t|^2 \big] = \E\big[ \big< \bar M^{f,K} \big>_t \big] \\
\le c_2 ||f||_{C^{1,1}}^2 \frac1K \int_0^t \Big( \E\big[ (1,\bar S^K_u) \big] + \E\big[ (1,\bar S^K_u)^2 \big] \Big) du. 
\end{multline*}
Notice that 
$$\E\big[ (1,\bar S^K_u) \big] + \E\big[ (1,\bar S^K_u)^2 \big] 
\le \E\big[ \bar X^K_u \big] + \E\big[ (\bar X^K_u)^2 \big] \le c_T$$
from the previous section. 
Therefore, (J2a) follows by Markov's inequality. 

Now, let $\zeta\le\delta$ and $\tau^K$ be a sequence of stopping times bounded by $T$. 
We have 
$$\E\big[ |\bar V^{f,K}_{\tau^K+\zeta\wedge T} - \bar V^{f,K}_{\tau^K}| \big] 
\le c_3 ||f||_{C^{1,1}} \E\bigg[ \int_{\tau^K}^{\tau^K+\zeta\wedge T} \big( (1,\bar S^K_u) + (1,\bar S^K_u)^2 \big) du \bigg]$$ 
and 
\begin{align*}
&\E\bigg[ \int_{\tau^K}^{\tau^K+\zeta\wedge T} \big( (1,\bar S^K_u) + (1,\bar S^K_u)^2 \big) du \bigg] 
\le \int_0^\zeta \E\big[ (1,\bar S^K_{\tau^K+u\wedge T}) + (1,\bar S^K_{\tau^K+u\wedge T})^2 \big] du \\
&\q\q\q\le \int_0^\zeta \Big( \E\Big[ \sup_{r\le T} (1,\bar S^K_r) \Big] + \E\Big[ \sup_{r\le T} (1,\bar S^K_r)^2 \Big] \Big) du \\
&\q\q\q\le \int_0^\zeta \Big( \E\Big[ \sup_{r\le T} \bar X^K_r \Big] + \E\Big[ \sup_{r\le T} (\bar X^K_r)^2 \Big] \Big) du 
\ \le\ c_T \delta,
\end{align*}
where the last inequality follows from the results in the previous section. 
Similarly, 
\begin{multline*}
\E\big[ |\big<\bar M^{f,K}\big>_{\tau^K+\zeta\wedge T} - \big<\bar M^{f,K}\big>_{\tau^K}| \big] \\
\le c_4 ||f||_{C^{1,1}}^2 \frac1K \E\bigg[ \int_{\tau^K}^{\tau^K+\zeta\wedge T} \big( (1,\bar S^K_u) + (1,\bar S^K_u)^2 \big) du \bigg] 
\le c'_T\delta ||f||_{C^{1,1}}^2.
\end{multline*}
(J2b) thus follows by Markov's inequality.
\end{proof}

We see that the martingale $\bar M^{f,K}$ converges to 0, since its predictable quadratic variation vanishes. 

\begin{proof} [of Theorem \ref{T:MonogamyLLN} -- LLN]
Suppose that $\mathscr{S}$ is a limit point of $\bar S^K$. 
Note that for any bounded function $f$ and reproduction parameter $q = h^\F, h^\M, h^\FM, bm^\F, bm^\M$, 
\begin{multline} \label{E:fqS}
|(f q^K_{\bar S^K_u}, \bar S^K_u) - (f q^\infty_{\mathscr{S}_u}, \mathscr{S}_u)| 
\le |(f (q^K_{\bar S^K_u} - q^\infty_{\mathscr{S}_u}), \bar S^K_u)| + |(f q^\infty_{\mathscr{S}_u}, \bar S^K_u- \mathscr{S}_u)| \\ 
\le ||f||_\infty ||q^K_{\bar S^K_u} - q^\infty_{\mathscr{S}_u}||_\infty (1,\bar S^K_u) 
+ ||f||_\infty ||q^\infty_{\mathscr{S}_u}||_\infty ||\bar S^K_u- \mathscr{S}_u|| 
\end{multline}
converges to zero as $K\to\infty$. 
Indeed, the first term on the right hand side vanishes since
$$||q^K_{\bar S^K_u} - q^\infty_{\mathscr{S}_u}||_\infty 
\le ||q^K_{\bar S^K_u} - q^K_{\mathscr{S}_u}||_\infty + ||q^K_{\mathscr{S}_u} - q^\infty_{\mathscr{S}_u}||_\infty 
\to 0$$
due to (C1') and (C2'); clearly, the last term in \eqref{E:fqS} also vanishes. 
Similarly, 
\begin{align*}
&\big| \big(\big( g(\Cdot,\ast) K\rho^K_{\bar S^K_u}(\Cdot,\ast) \mathbf{1}_\F, \bar S^K_u\big)_{\Cdot} \mathbf{1}_\M, \bar S^K_u\big)_\ast 
- \big(\big( g(\Cdot,\ast) \rho^\infty_{\mathscr{S}_u}(\Cdot,\ast) \mathbf{1}_\F, \mathscr{S}_u\big)_{\Cdot} \mathbf{1}_\M,\mathscr{S}_u\big)_\ast \big| \\
&\q\le \big| \big(\big( g(\Cdot,\ast) K\rho^K_{\bar S^K_u}(\Cdot,\ast) \mathbf{1}_\F, \bar S^K_u\big)_{\Cdot} \mathbf{1}_\M, \bar S^K_u\big)_\ast 
- \big(\big( g(\Cdot,\ast) \rho^\infty_{\mathscr{S}_u}(\Cdot,\ast) \mathbf{1}_\F, \bar S^K_u\big)_{\Cdot} \mathbf{1}_\M, \bar S^K_u\big)_\ast \big| \\
&\q\q+ \big| \big(\big( g(\Cdot,\ast) \rho^\infty_{\mathscr{S}_u}(\Cdot,\ast) \mathbf{1}_\F, \bar S^K_u\big)_{\Cdot} \mathbf{1}_\M, \bar S^K_u\big)_\ast 
- \big(\big( g(\Cdot,\ast) \rho^\infty_{\mathscr{S}_u}(\Cdot,\ast) \mathbf{1}_\F, \mathscr{S}_u\big)_{\Cdot} \mathbf{1}_\M, \bar S^K_u\big)_\ast \big| \\
&\q\q+ \big| \big(\big( g(\Cdot,\ast) \rho^\infty_{\mathscr{S}_u}(\Cdot,\ast) \mathbf{1}_\F, \mathscr{S}_u\big)_{\Cdot} \mathbf{1}_\M, \bar S^K_u\big)_\ast 
- \big(\big( g(\Cdot,\ast) \rho^\infty_{\mathscr{S}_u}(\Cdot,\ast) \mathbf{1}_\F, \mathscr{S}_u\big)_{\Cdot} \mathbf{1}_\M, \mathscr{S}_u\big)_\ast \big| \\
&\q\le ||g||_\infty ||K\rho^K_{\bar S^K_u} - \rho^\infty_{\mathscr{S}_u}||_\infty (1,\bar S^K_u)^2 
+ ||g||_\infty ||\rho^\infty_{\mathscr{S}_u}||_\infty ||\bar S^K_u-\mathscr{S}_u|| (1,\bar S^K_u) \\
&\q\q+ ||g||_\infty ||\rho^\infty_{\mathscr{S}_u}||_\infty (1,\mathscr{S}_u) ||\bar S^K_u-\mathscr{S}_u||, 
\end{align*}
can be shown vanishing as $K\to\infty$ using the trick as above with (C1') and (C2'), and that $\mathscr{S}_u$ is a limit point of $\bar S^K_u$. 
Thus limit point of $\bar S^K$ satisfies \eqref{E:SbarInftyMonogamy}. 

It remains to show that the limit is unique. 
This is done by considering \eqref{E:SbarInftyMonogamy} with test functions that depend also on time, $f(i,v,w,t)$. 
Now, fix $t\in\T$, for $u\le t$, take $f$ of the form $f(i,v,w,u) = \phi(i,v+t-u,w+t-u) = \widehat\Theta_{t-u} \phi(i,v,w)$ as in Proposition \ref{P:AltRepS}.
Then, $\mathscr{S}$ is shown to satisfy the following equation: 
\begin{align}
\notag&(\phi,\mathscr{S}_t) = (\widehat\Theta_t\phi,\mathscr{S}_0) 
- \int_0^t \big(\widehat\Theta_{t-u}\phi h^{\F,\infty}_{\mathscr{S}_u} \mathbf{1}_\F + \widehat\Theta_{t-u}\phi h^{\M,\infty}_{\mathscr{S}_u} \mathbf{1}_\M, \mathscr{S}_u\big) du \\
\notag&\q+ \int_0^t \Big(\Big[\big(\widehat\Theta_{t-u}\phi(\M,\infty,\cdot)-\widehat\Theta_{t-u}\phi(\FM,\cdot,\cdot)\big) h^{\F,\infty}_{\mathscr{S}_u} + \big(\widehat\Theta_{t-u}\phi(\F,\cdot,\infty)-\widehat\Theta_{t-u}\phi(\FM,\cdot,\cdot)\big) h^{\M,\infty}_{\mathscr{S}_u} \Big] \mathbf{1}_\FM, \mathscr{S}_u\Big) du \\
\notag&\q+ \int_0^t \Big(\big(\widehat\Theta_{t-u}\phi(\F,\cdot,\infty)+\widehat\Theta_{t-u}\phi(\M,\infty,\cdot)-\widehat\Theta_{t-u}\phi(\FM,\cdot,\cdot)\big) h^{\FM,\infty}_{\mathscr{S}_u} \mathbf{1}_\FM, \mathscr{S}_u\Big) du \\
\notag&\q+ \int_0^t \Big( \big[\widehat\Theta_{t-u}\phi(\F,0,\infty)m^{\F,\infty}_{\mathscr{S}_u} + \widehat\Theta_{t-u}\phi(\M,\infty,0)m^{\M,\infty}_{\mathscr{S}_u}\big] b^\infty_{\mathscr{S}_u} (\mathbf{1}_\FM+\mathbf{1}_\F), \mathscr{S}_u\Big) du \\
&\q- \int_0^t \Big(\Big( \big[\widehat\Theta_{t-u}\phi(\F,\Cdot,\infty)+\widehat\Theta_{t-u}\phi(\M,\infty,\ast)-\widehat\Theta_{t-u}\phi(\FM,\Cdot,\ast)\big] \rho^\infty_{\mathscr{S}_u}(\Cdot,\ast) \mathbf{1}_\F, \mathscr{S}_u\Big)_{\Cdot} \mathbf{1}_\M, \mathscr{S}_u\Big)_\ast du. \label{E:SbarInftyAltMonogamy}
\end{align}
The uniqueness of the limit is then achieved by considering two processes $\mathscr{S}^1$ and $\mathscr{S}^2$ that are both solutions to \eqref{E:SbarInftyAltMonogamy} with the same initial point $\mathscr{S}^1_0=\mathscr{S}^2_0$. 
Then we show that $|(\phi,\mathscr{S}^1_t-\mathscr{S}^2_t)|$ is bounded by $c_T ||\phi||_\infty ||\mathscr{S}^1_u-\mathscr{S}^2_u||$. 
Therefore, $\mathscr{S}^1=\mathscr{S}^2$ by Gronwall's inequality. 
\end{proof}

\subsection{Proof of the Central Limit Theorem (serial monogamy mating system)} \label{S:Appendix-Monogamy-CLT}

The mechanism in proving the Central Limit Theorem in this serial monogamy mating system would be slightly different from that in Section \ref{S:CLTProof} due to the iterative bracket, coming from the coupling. 
To overcome the issue arose from the coupling, we need the stopping times 
$$\tau^K_N = \inf\{t\in\T: \bar X^K_t>N\}.$$

As before, we obtain an alternative representation equation for $Z^K$. 
This is done by extending \eqref{E:fZKMonogamy} to have test function depending on time, $f = f(i,v,w,t)$ 
and then, with fixed $t$, taking $f(i,v,w,u) = \phi(i,v+t-u,w+t-u) = \Theta_{t-u}\phi(i,v,w)$ for $u\le t$ and $\phi \in W^{-4}$ as in Remark \ref{R:ShiftOp}. 
We then have  
{\footnotesize
\begin{align} 
\notag&(\phi,Z_t^K) = (\Theta_t\phi,Z_0^K) \\
\notag&\q- \int_0^t \sqrt{K} \big(\Theta_{t-u}\phi (h^{\F,K}_{\bar S^K_u}-h^{\F,\infty}_{\bar S_u}) \mathbf{1}_\F + \Theta_{t-u}\phi (h^{\M,K}_{\bar S^K_u}-h^{\M,\infty}_{\bar S_u}) \mathbf{1}_\M, \bar S_u\big) du 
- \int_0^t \big(\Theta_{t-u}\phi h^{\F,K}_{\bar S^K_u} \mathbf{1}_\F + \Theta_{t-u}\phi h^{\M,K}_{\bar S^K_u} \mathbf{1}_\M, Z^K_u\big) du \\
\notag&\q+ \int_0^t \sqrt{K} \Big(\Big[\big(\Theta_{t-u}\phi(\M,\infty,\cdot)-\Theta_{t-u}\phi(\FM,\cdot,\cdot)\big) (h^{\F,K}_{\bar S^K_u}-h^{\F,\infty}_{\bar S_u}) + \big(\Theta_{t-u}\phi(\F,\cdot,\infty)-\Theta_{t-u}\phi(\FM,\cdot,\cdot)\big) (h^{\M,K}_{\bar S^K_u}-h^{\M,\infty}_{\bar S_u}) \Big] \mathbf{1}_\FM, \bar S_u\Big) du \\
\notag&\q+ \int_0^t \Big(\Big[\big(\Theta_{t-u}\phi(\M,\infty,\cdot)-\Theta_{t-u}\phi(\FM,\cdot,\cdot)\big) h^{\F,K}_{\bar S^K_u} + \big(\Theta_{t-u}\phi(\F,\cdot,\infty)-\Theta_{t-u}\phi(\FM,\cdot,\cdot)\big) h^{\M,K}_{\bar S^K_u} \Big] \mathbf{1}_\FM, Z^K_u\Big) du \\
\notag&\q+ \int_0^t \sqrt{K} \Big(\big(\Theta_{t-u}\phi(\F,\cdot,\infty)+\Theta_{t-u}\phi(\M,\infty,\cdot)-\Theta_{t-u}\phi(\FM,\cdot,\cdot)\big) (h^{\FM,K}_{\bar S^K_u}-h^{\FM,\infty}_{\bar S_u}) \mathbf{1}_\FM, \bar S_u\Big) du \\
\notag&\q+ \int_0^t \Big(\big(\Theta_{t-u}\phi(\F,\cdot,\infty)+\Theta_{t-u}\phi(\M,\infty,\cdot)-\Theta_{t-u}\phi(\FM,\cdot,\cdot)\big) h^{\FM,K}_{\bar S^K_u} \mathbf{1}_\FM, Z^K_u\Big) du \\
\notag&\q+ \int_0^t \sqrt{K} \Big( \big[\Theta_{t-u}\phi(\F,0,\infty) (b^K_{\bar S^K_u}m^{\F,K}_{\bar S^K_u} - b^\infty_{\bar S_u}m^{\F,\infty}_{\bar S_u}) + \Theta_{t-u}\phi(\M,\infty,0) (b^K_{\bar S^K_u}m^{\M,K}_{\bar S^K_u} - b^\infty_{\bar S_u}m^{\M,\infty}_{\bar S_u}) \big] (\mathbf{1}_\FM+\mathbf{1}_\F), \bar S_u\Big) du \\
\notag&\q+ \int_0^t \Big( \big[\Theta_{t-u}\phi(\F,0,\infty)m^{\F,K}_{\bar S^K_u} + \Theta_{t-u}\phi(\M,\infty,0)m^{\M,K}_{\bar S^K_u}\big] b^K_{\bar S^K_u} (\mathbf{1}_\FM+\mathbf{1}_\F), Z^K_u\Big) du \\
\notag&\q- \int_0^t \sqrt{K} \Big(\Big( [\Theta_{t-u}\phi(\F,\Cdot,\infty)+\Theta_{t-u}\phi(\M,\infty,\ast)-\Theta_{t-u}\phi(\FM,\Cdot,\ast)] (K \rho^K_{\bar S^K_u}(\Cdot,\ast) - \rho^\infty_{\bar S_u}(\Cdot,\ast)) \mathbf{1}_\F, \bar S_u\Big)_{\Cdot} \mathbf{1}_\M, \bar S_u\Big)_\ast du \\
\notag&\q- \int_0^t \Big(\Big( [\Theta_{t-u}\phi(\F,\Cdot,\infty)+\Theta_{t-u}\phi(\M,\infty,\ast)-\Theta_{t-u}\phi(\FM,\Cdot,\ast)] K \rho^K_{\bar S^K_u}(\Cdot,\ast) \mathbf{1}_\F, Z^K_u\Big)_{\Cdot} \mathbf{1}_\M, \bar S_u\Big)_\ast du \\
\notag&\q- \int_0^t \Big(\Big( [\Theta_{t-u}\phi(\F,\Cdot,\infty)+\Theta_{t-u}\phi(\M,\infty,\ast)-\Theta_{t-u}\phi(\FM,\Cdot,\ast)] K \rho^K_{\bar S^K_u}(\Cdot,\ast) \mathbf{1}_\F, \bar S_u\Big)_{\Cdot} \mathbf{1}_\M, Z^K_u\Big)_\ast du \\
&\q+ \int_0^t (\Theta_{t-u}\phi, d\tilde M^K_u), \label{E:phiZKMonogamy}
\end{align}
}
where $\tilde M$ is the measure such that $(f,\tilde M^K_t) = \tilde M^{f,K}_t$. 
This representation helps establishing the next result, which in turn is used in proving tightness of the sequence $Z^K$. 

\begin{proposition} \label{P:PZK>R}
For any $\epsilon>0$, there exists $R>0$ such that for all $K$, 
$$\P(||Z^K_t||_{W^{-2}} >R) \le \epsilon.$$
\end{proposition}

\begin{proof}
Let $N > \sup_{t\in\T} \bar X^\infty_t$, where $\bar X^\infty_t = (\1_\F+\1_\M+2\1_\FM,\bar S_t)$. 
Define stopping time $\tau^K_N = \inf\{t\in\T: \bar X^K_t >N\}$. 
Then, we have 
\begin{align}
\notag\P(||Z^K_t||_{W^{-2}} > R) 
&= \P(||Z^K_t||_{W^{-2}} > R, \tau^K_N > T) + \P(||Z^K_t||_{W^{-2}} > R, \tau^K_N \le T) \\ 
&\le \frac1R \E[ ||Z^K_t||_{W^{-2}} \1_{\tau^K_N>T}] + \P(\tau^K_N \le T). 
\label{E:PZ>R}
\end{align}
Observe that 
$$\P(\tau^K_N \le T) 
= \P\Big(\sup_{t\le T} \bar X^K_t > N\Big) 
\le \frac1N \E\Big[\sup_{t\le T} \bar X^K_t \Big],$$
thus, by \eqref{E:EsupXBnd}, for any $\epsilon$, there exists $N$ such that $\P(\tau^K_N \le T) \le \epsilon/2$ for all $K$. 

For the first term in \eqref{E:PZ>R}, we proceed from representation in \eqref{E:phiZK}. 
For $\phi\in W^2$, as $||\Theta_r\phi||_{W^2} \le c||\phi||_{W^2}$ for $r\in\T$, we have the following: 
\begin{itemize}[leftmargin=*]
\item [$\bullet$]
$|(\Theta_t\phi,Z^K_0)| \le ||\Theta_t\phi||_{W^2} ||Z^K_0||_{W^{-2}} \le c ||\phi||_{W^2} ||Z^K_0||_{W^{-2}}$,
\item [$\bullet$]
$\sqrt{K} \big|\big(\Theta_{t-u}\phi (q^K_{\bar S^K_u} - q^\infty_{\bar S_u}), \bar S_u\big) \big| 
\le ||\Theta_t\phi||_\infty ||\sqrt{K}(q^K_{\bar S^K_u} - q^\infty_{\bar S_u})||_\infty (1,\bar S_u)
\le c ||\phi||_{W^2} (1+||Z^K_u||_{W^{-4}}) \le c' ||\phi||_{W^2} (1+||Z^K_u||_{W^{-2}})$,
\item [$\bullet$]
$\big|\big(\Theta_{t-u}\phi q^K_{\bar S^K_u}, Z^K_u\big)\big| \le ||\Theta_{t-u}\phi||_{W^2} ||q^K_{\bar S^K_u}||_{C^2} ||Z^K_u||_{W^{-2}} \le c ||\phi||_{W^2} ||Z^K_u||_{W^{-2}}$, 
\item [$\bullet$]
$\sqrt{K} \big|\big(\big( [\Theta_{t-u}\phi(\F,\Cdot,\infty)+\Theta_{t-u}\phi(\M,\infty,\ast)-\Theta_{t-u}\phi(\FM,\Cdot,\ast)] (K \rho^K_{\bar S^K_u}(\Cdot,\ast) - \rho^\infty_{\bar S_u}(\Cdot,\ast)) \mathbf{1}_\F, \bar S_u\big)_{\Cdot} \mathbf{1}_\M, \bar S_u\big)_\ast \big| 
\le c ||\phi||_{W^2} (1+||Z^K_u||_{W^{-4}}) (1,\bar S_u)^2 
\le c' ||\phi||_{W^2} (1+||Z^K_u||_{W^{-2}})|$,
\item [$\bullet$]
$\big|\big(\big( [\Theta_{t-u}\phi(\F,\Cdot,\infty)+\Theta_{t-u}\phi(\M,\infty,\ast)-\Theta_{t-u}\phi(\FM,\Cdot,\ast)] K \rho^K_{\bar S^K_u}(\Cdot,\ast) \mathbf{1}_\F, Z^K_u\big)_{\Cdot} \mathbf{1}_\M, \bar S_u\big)_\ast \big| 
\le c ||\Theta_t\phi||_{W^2} ||K\rho^K_{\bar S^K_u}||_{C^2} ||Z^K_u||_{W^{-2}} (1,\bar S_u) 
\le c' ||\phi||_{W^2} ||Z^K_u||_{W^{-2}}$,
\item [$\bullet$]
$\big| \big(\big( [\Theta_{t-u}\phi(\F,\Cdot,\infty)+\Theta_{t-u}\phi(\M,\infty,\ast)-\Theta_{t-u}\phi(\FM,\Cdot,\ast)] K \rho^K_{\bar S^K_u}(\Cdot,\ast) \mathbf{1}_\F, \bar S^K_u\big)_{\Cdot} \mathbf{1}_\M, Z^K_u\big)_\ast \big| 
= \big| \big(\big( [\Theta_{t-u}\phi(\F,\Cdot,\infty)+\Theta_{t-u}\phi(\M,\infty,\ast)-\Theta_{t-u}\phi(\FM,\Cdot,\ast)] K \rho^K_{\bar S^K_u}(\Cdot,\ast) \mathbf{1}_\M, Z^K_u\big)_\ast \mathbf{1}_\F, \bar S^K_u\big)_{\Cdot} \big|
\le c ||\phi||_{W^2} ||Z^K_u||_{W^{-2}} (1,\bar S^K_u)$,
\item [$\bullet$]
$\int_0^t (\Theta_{t-u}\phi, d\tilde M^K_u) \le ||\phi||_{W^2} \big|\big|\int_0^t \Theta_{t-u}^* d\tilde M^K_u \big|\big|_{W^{-2}}$,
\end{itemize}
where we write $\int_0^t \Theta_{t-u}^* d\tilde M^K_u$ for the operator such that $(f, \int_0^t \Theta_{t-u}^* d\tilde M^K_u) = \int_0^t (\Theta_{t-u}f, d\tilde M^K_u)$.
Therefore, 
$$||Z^K_t||_{W^{-2}} \le c \bigg( ||Z^K_0||_{W^{-2}} + \int_0^t \big( 1+ ||Z^K_u||_{W^{-2}} + ||Z^K_u||_{W^{-2}} (1,\bar S^K_u) \big) du \bigg) + \Big|\Big|\int_0^t \Theta_{t-u}^* d\tilde M^K_u \Big|\Big|_{W^{-2}}.$$ 
Multiply this by $\1_{\tau^K_N>T}$ and take expectation, we then obtain 
\begin{align*}
\E\Big[ ||Z^K_t||_{W^{-2}} \1_{\tau^K_N>T}\Big] 
&\le c_1 \bigg( ||Z^K_0||_{W^{-2}} \E\Big[\1_{\tau^K_N>T}\Big] \\
&\q\q+ \int_0^t \E \Big[ \big( 1+ ||Z^K_u||_{W^{-2}} + ||Z^K_u||_{W^{-2}} (1,\bar S^K_u) \big) \1_{\tau^K_N>T} \Big] du \bigg) \\
&\q\q\q\q+ \E\Big[ \Big|\Big|\int_0^t \Theta_{t-u}^* d\tilde M^K_u \Big|\Big|_{W^{-2}} \1_{\tau^K_N>T} \Big] \\
&\le c_2 \bigg( ||Z^K_0||_{W^{-2}} + t + \int_0^t \E \big[ \big( ||Z^K_u||_{W^{-2}} + ||Z^K_u||_{W^{-2}} N \big) \1_{\tau^K_N>T} \big] du \bigg) \\
&\q\q\q\q+ \E\Big[ \Big|\Big|\int_0^t \Theta_{t-u}^* d\tilde M^K_u \Big|\Big|_{W^{-2}} \Big] 
\end{align*}
as $(1,\bar S^K_u) \le \bar X^K_u$ which is less than $N$ on the set $\{\tau^K_N>T\}$. 
To deal with the martingale term, 
let $(p_l)_{l\ge1}$ denote a complete orthonormal basis of $W^2$. 
Then, for $r\le t$, 
\begin{multline*}
\E\Big[\big|\big|\int_0^r \Theta_{t-u}^* d\tilde M^K_u \big|\big|_{W^{-2}}^2 \Big] 
= \E\bigg[ \sum_{l\ge1} \Big( \int_0^r (\Theta_{t-u}p_l, d\tilde M^K_u) \Big)^2 \bigg] \\
\le \sum_{l\ge1} \E\bigg[ \Big< \int_0^\cdot (\Theta_{t-u}p_l, d\tilde M^K_u) \Big>_r \bigg] 
\le \E \bigg[ c\int_0^r \big((1,\bar S^K_u) + (1,\bar S^K_u)^2\big) du \bigg] 
\le c_T, 
\end{multline*}
where the third line is due to 
$\big| \sum_{l\ge1} p_l(s_1)p_l(s_2) \big| \le c$ for any $s_1,s_2 \in W^2$ 
and the results in Section \ref{S:Appendix-Monogamy-TPS}. 
Take $r=t$, 
$\E\big[||\int_0^t \Theta_{t-u}^* d\tilde M^K_u ||_{W^{-2}} \big]^2 \le \E\big[||\int_0^t \Theta_{t-u}^* d\tilde M^K_u ||_{W^{-2}}^2 \big] \le c_T$. 
Hence, 
$$\E[ ||Z^K_t||_{W^{-2}} \1_{\tau^K_N>T}] \le c_3 \bigg( ||Z^K_0||_{W^{-2}} + T + (1+N) \int_0^t \E \big[ ||Z^K_u||_{W^{-2}} \1_{\tau^K_N>T} \big] du + 1 \bigg)$$
and by Gronwall's inequality, 
$$\E[ ||Z^K_t||_{W^{-2}} \1_{\tau^K_N>T}] \le c_3 \big(||Z^K_0||_{W^{-2}} + T + 1\big) e^{c_3(1+N)T} \le c_T(N).$$
Thus, for any $\epsilon$ and $N$, there exists $R$ such that 
$\frac1R \E[ ||Z^K_t||_{W^{-2}} \1_{\tau^K_N>T}] \le \epsilon/2$ for all $K$. 

The assertion then follows.
\end{proof}

\begin{proposition} \label{P:ZW-3}
Let $N > \sup_{t\in\T} \bar X^\infty_t$, where $\bar X^\infty_t = (\1_\F+\1_\M+2\1_\FM,\bar S_t)$. 
Define stopping time $\tau^K_N = \inf\{t\in\T: \bar X^K_t >N\}$. 
Then, 
$$\E\Big[ \sup_{t\in\T} ||Z^K_t||_{W^{-3}} \1_{\tau^K_N>T} \Big] \le c_T(N).$$
\end{proposition}

\begin{proof}
Let $f\in W^3$. 
From \eqref{E:fZKMonogamy}, take absolute value and bound each term on the right hand side, similar as in Proposition \ref{P:PZK>R}. With the inclusions of the spaces, $W^{-2} \hookrightarrow W^{-3} \hookrightarrow W^{-4}$, and the boundedness of the model parameters, we have
\begin{multline*}
|(f,Z^K_t)| \le ||f||_{W^3} \Big( ||Z^K_0||_{W^{-3}} \\
+ c_1 \int_0^t \big( 1+ ||Z^K_u||_{W^{-2}} + ||Z^K_u||_{W^{-2}} (1,\bar S^K_u) \big) du + ||\tilde M^K_t||_{W^{-3}} \Big). 
\end{multline*}
Thus, 
$$||Z^K_t||_{W^{-3}} \le ||Z^K_0||_{W^{-3}} + c_1 \int_0^t \Big( 1+ ||Z^K_u||_{W^{-2}} \big(1+ (1,\bar S^K_u)\big) \Big) du + ||\tilde M^K_t||_{W^{-3}}$$
and 
\begin{multline*}
\E\Big[ \sup_{t\in\T} ||Z^K_t||_{W^{-3}} \1_{\tau^K_N>T} \Big] \le ||Z^K_0||_{W^{-3}} \\
+ c_1 \Big( T + (1+N) \int_0^T \E\big[ ||Z^K_u||_{W^{-2}} \1_{\tau^K_N>T} \big] du \Big) 
+ \E\Big[ \sup_{t\in\T} ||\tilde M^K_t||_{W^{-3}} \Big].
\end{multline*}
Now, let $(p_l)_{l\ge1}$ denote a complete orthonormal basis of $W^3$. Then, 
\begin{multline*}
\E\Big[ \sup_{t\in\T} ||\tilde M^K_t||_{W^{-3}} \Big] 
= \E\Big[ \sup_{t\in\T} \sum_{l\ge1} (\tilde M^{p_l,K}_t)^2 \Big] 
\le \sum_{l\ge1} \E\Big[ \sup_{t\in\T} (\tilde M^{p_l,K}_t)^2 \Big] \\
\le 4 \sum_{l\ge1} \E\Big[ \big<\tilde M^{p_l,K} \big>_T \Big] 
\le c \int_0^T \E\big[ (1,\bar S^K_u) + (1,\bar S^K_u)^2 \big] du 
\le c_T. 
\end{multline*}
From the proof of Proposition \ref{P:PZK>R}, we also have 
$\E\big[ ||Z^K_u||_{W^{-2}} \1_{\tau^K_N>T} \big] \le c_T(N)$. 
The assertion thus follows. 
\end{proof}

\begin{proposition}
The sequence $Z^K$ is tight in $\D(\T,W^{-4})$. 
\end{proposition}

\begin{proof}
In a similar way as in Proposition \ref{P:TightZM}, 
(T1) follows from Proposition \ref{P:PZK>R}. 
%Let $B_{W^{-2}}(R) = \{\mu\in W^{-2}: ||\mu||_{W^{-2}} \le R\}$. 
%Since $W^{-2}$ is Hilbert-Schmidt embedded in $W^{-4}$, $B_{W^{-2}}(R)$ is compact in $W^{-4}$. 
%Proposition \ref{P:PZK>R} shows that for any $\epsilon>0$, there exists $R$ such that 
%$\P(Z^K_t \notin B_{W^{-2}}(R)) \le \epsilon$ for all $K\ge1$. 

For (T2a), observe that 
$$\P( ||\tilde V^K_{\tau^K+\zeta\wedge T} - \tilde V^K_{\tau^K}||_{W^{-4}} > \epsilon_1) 
\le \P( ||\tilde V^K_{\tau^K+\zeta\wedge T} - \tilde V^K_{\tau^K}||_{W^{-4}} > \epsilon_1, \tau^K_N >T) + \P(\tau^K_N\le T).$$
As shown in the proof of Proposition \ref{P:PZK>R}, for any $\epsilon_2$, there exists $N$ such that $\P(\tau^K_N\le T) < \epsilon_2/2$. 
Using same technique as before, we can also show that 
$$||\tilde V^K_{\tau^K+\zeta\wedge T} - \tilde V^K_{\tau^K}||_{W^{-4}} 
\le c \int_{\tau^K}^{\tau^K+\zeta\wedge T} \Big(1+ ||Z^K_u||_{W^{-3}} \big(1+(1,\bar S^K_u)\big) \Big) du$$ 
and thus 
\begin{align*}
\E\big[ ||\tilde V^K_{\tau^K+\zeta\wedge T} - \tilde V^K_{\tau^K}||_{W^{-4}} \1_{\tau^K_N >T} \big] 
&\le c\bigg( \zeta + \E\bigg[ \int_{\tau^K}^{\tau^K+\zeta\wedge T} ||Z^K_u||_{W^{-3}} \big(1+(1,\bar S^K_u)\big) du  \1_{\tau^K_N >T} \bigg] \bigg) \\
&\le c\bigg( \zeta + (1+N) \E\bigg[ \int_0^\zeta ||Z^K_{\tau^K+u\wedge T}||_{W^{-3}} du  \1_{\tau^K_N >T} \bigg] \bigg) \\
&\le c\bigg( \zeta + (1+N) \int_0^\zeta \E\Big[ \sup_{u\le T} ||Z^K_u||_{W^{-3}}  \1_{\tau^K_N >T} \bigg] du \bigg) \\
&\le c_T(N) \delta, 
\end{align*}
where the last inequality follows from Proposition \ref{P:ZW-3}. 
Therefore, 
\begin{align*}
\P( ||\tilde V^K_{\tau^K+\zeta\wedge T} - \tilde V^K_{\tau^K}||_{W^{-4}} > \epsilon_1, \tau^K_N >T) 
&\le \frac1{\epsilon_1} \E\big[ ||\tilde V^K_{\tau^K+\zeta\wedge T} - \tilde V^K_{\tau^K}||_{W^{-4}} \1_{\tau^K_N >T} \big] \\
&\le \frac{c_T(N) \delta}{\epsilon_1},
\end{align*}
and for any $\epsilon_1$, $\epsilon_2$ and $N$, there exists $\delta$ such that the probabilty does not exceed $\epsilon_2/2$. 

For (T2b), let $(p_l)_{l\ge1}$ denote a complete orthonormal basis of $W^4$. 
Note that $||\tilde M^K_t||_{W^{-4}}^2 = \sum_{l\ge1} (\tilde M^{p_l,K}_t)^2$ and 
$\big<\big<\tilde M^K\big>\big>_t = \sum_{l\ge1} \big<\tilde M^{p_l,K}\big>_t$. 
So, 
\begin{align*}
\E\big[\big| \big<\big<\tilde M^K\big>\big>_{\tau^K+\zeta\wedge T} - \big<\big<\tilde M^K\big>\big>_{\tau^K}\big|\big] 
&= \E\bigg[\bigg| \sum_{l\ge1} \big<\tilde M^{p_l,K}\big>_{\tau^K+\zeta\wedge T} - \sum_{l\ge1} \big<\tilde M^{p_l,K}\big>_{\tau^K} \bigg|\bigg] \\
&\le \E\bigg[ c \int_{\tau^K}^{\tau^K+\zeta\wedge T} \big((1,\bar S^K_u) + (1,\bar S^K_u)^2\big) du \bigg] \\
&\le c \delta \bigg( \E\bigg[ \sup_{u\le T} (1,\bar S^K_u) \bigg] + \E\bigg[ \sup_{u\le T} (1,\bar S^K_u)^2 \bigg] \bigg) \\
&\le c_T \delta. 
\end{align*}
\end{proof}

\begin{proposition}
$Z^K$ and $\tilde M^K$ are C-tight, that is, all limit points of $Z^K$ and $\tilde M^K$ are in $\mathbb{C}(\T, W^{-4})$. 
\end{proposition}

\begin{proof}
%By \cite{JacShi03} Proposition VI.3.26, a sequence $(Z^K)$ is C-tight if and only if the sequence is tight and for all $N\ge1$ and $\epsilon>0$, $\lim_K \P( \sup_{t\le N} |\Delta Z^K_t| > \epsilon) = 0$. 
Note that $Z^K$ jumps when $S^K$ jumps, which jumps when there is an event associated with birth, death or marriage. 
Let $f\in W^4$. Then, assuming that the number of offspring is bounded by $\Xi$ which has finite mean and variance, 
\begin{align*}
&|(f,\Delta Z^K_t)| = \frac1{\sqrt{K}} |(f,S^K_t-S^K_{t-})| \\
&\q\le \frac1{\sqrt{K}} \Big( |f(1,0,\infty)| \Xi + |f(2,\infty,0)| \Xi 
+ \sup_v |f(1,v,\infty)| + \sup_w |f(2,\infty,w)| + \sup_{v,w} |f(3,v,w)| \Big) \\
&\q\le \frac1{\sqrt{K}} c ||f||_\infty (1+\Xi)
\end{align*}
and 
$$\E\big[ \sup_{t\le u} ||\Delta Z^K_t||_{W^{-4}} \bigg] \le \frac{c}{\sqrt{K}} (1+\E[\Xi]).$$ 
It follows that, for any $u$ and $\epsilon>0$, 
$$\P\bigg(\sup_{t\le u} ||\Delta Z^K_t||_{W^{-4}} >\epsilon \bigg) 
\le \frac1\epsilon \E\big[ \sup_{t\le u} ||\Delta Z^K_t||_{W^{-4}} \bigg] 
\le \frac{c}{\epsilon\sqrt{K}} (1+\E[\Xi]),$$
which converges to zero as $K$ tends to infinity. 
This, together with the tightness of $Z^K$, shows that the sequence is C-tight (see e.g. \cite[Proposition VI.3.26]{JacShi03}). 

Similarly, $\tilde M^K$ is C-tight, since $\tilde M^K$ has the same jumps as $Z^K$. 
\end{proof}

\begin{proposition}
$\tilde M^K$ converges in $\D(\T,W^{-4})$ to $\tilde M^\infty$, defined such that for each $f\in W^4$, $(f, \tilde M^\infty) = \tilde M^{f,\infty}$ is a martingale with predictable quadratic variation \eqref{E:MinftyPQV}.
\end{proposition}

\begin{proof}
See proof of Proposition \ref{P:MInfty}
\end{proof}

Finally, Theorem \ref{T:CLTMonogamy} follows. 

\begin{proof} [of Theorem \ref{T:CLTMonogamy}]
Every limit point $\mathscr{Z}$  of $Z^K$ satisfies, for $\phi\in W^4$, 
{\footnotesize
\begin{align}
\notag&(\phi,\mathscr{Z}_t) = (\Theta_t\phi,\mathscr{Z}_0) \\
\notag&\q- \int_0^t \big(\Theta_{t-u}\phi \partial_Sh^{\F,\infty}_{\bar S_u}(\mathscr{Z}_u) \mathbf{1}_\F + \Theta_{t-u}\phi \partial_Sh^{\M,\infty}_{\bar S_u}(\mathscr{Z}_u) \mathbf{1}_\M, \bar S_u\big) du 
- \int_0^t \big(\Theta_{t-u}\phi h^{\F,\infty}_{\bar S_u} \mathbf{1}_\F + \Theta_{t-u}\phi h^{\M,\infty}_{\bar S_u} \mathbf{1}_\M, \mathscr{Z}_u\big) du \\
\notag&\q+ \int_0^t \Big(\Big[\big(\Theta_{t-u}\phi(\M,\infty,\cdot)-\Theta_{t-u}\phi(\FM,\cdot,\cdot)\big) \partial_Sh^{\F,\infty}_{\bar S_u}(\mathscr{Z}_u) + \big(\Theta_{t-u}\phi(\F,\cdot,\infty)-\Theta_{t-u}\phi(\FM,\cdot,\cdot)\big) \partial_Sh^{\M,\infty}_{\bar S_u}(\mathscr{Z}_u) \Big] \mathbf{1}_\FM, \bar S_u\Big) du \\
\notag&\q+ \int_0^t \Big(\Big[\big(\Theta_{t-u}\phi(\M,\infty,\cdot)-\Theta_{t-u}\phi(\FM,\cdot,\cdot)\big) h^{\F,\infty}_{\bar S_u} + \big(\Theta_{t-u}\phi(\F,\cdot,\infty)-\Theta_{t-u}\phi(\FM,\cdot,\cdot)\big) h^{\M,\infty}_{\bar S_u} \Big] \mathbf{1}_\FM, \mathscr{Z}_u\Big) du \\
\notag&\q+ \int_0^t \Big(\big(\Theta_{t-u}\phi(\F,\cdot,\infty)+\Theta_{t-u}\phi(\M,\infty,\cdot)-\Theta_{t-u}\phi(\FM,\cdot,\cdot)\big) \partial_Sh^{\FM,\infty}_{\bar S_u}(\mathscr{Z}_u) \mathbf{1}_\FM, \bar S_u\Big) du \\
\notag&\q+ \int_0^t \Big(\big(\Theta_{t-u}\phi(\F,\cdot,\infty)+\Theta_{t-u}\phi(\M,\infty,\cdot)-\Theta_{t-u}\phi(\FM,\cdot,\cdot)\big) h^{\FM,\infty}_{\bar S_u} \mathbf{1}_\FM, \mathscr{Z}_u\Big) du \\
\notag&\q+ \int_0^t \Big( \big[\Theta_{t-u}\phi(\F,0,\infty) \partial_S(b^\infty_{\bar S_u}m^{\F,\infty}_{\bar S_u})(\mathscr{Z}_u) + \Theta_{t-u}\phi(\M,\infty,0) \partial_S(b^\infty_{\bar S_u}m^{\M,\infty}_{\bar S_u})(\mathscr{Z}_u) \big] (\mathbf{1}_\FM+\mathbf{1}_\F), \bar S_u\Big) du \\
\notag&\q+ \int_0^t \Big( \big[\Theta_{t-u}\phi(\F,0,\infty)m^{\F,\infty}_{\bar S_u} + \Theta_{t-u}\phi(\M,\infty,0)m^{\M,\infty}_{\bar S_u}\big] b^\infty_{\bar S_u} (\mathbf{1}_\FM+\mathbf{1}_\F), \mathscr{Z}_u\Big) du \\
\notag&\q- \int_0^t \Big(\Big( [\Theta_{t-u}\phi(\F,\Cdot,\infty)+\Theta_{t-u}\phi(\M,\infty,\ast)-\Theta_{t-u}\phi(\FM,\Cdot,\ast)] \partial_S\rho^\infty_{\bar S_u}(\mathscr{Z}_u)(\Cdot,\ast) \mathbf{1}_\F, \bar S_u\Big)_{\Cdot} \mathbf{1}_\M, \bar S_u\Big)_\ast du \\
\notag&\q- \int_0^t \Big(\Big( [\Theta_{t-u}\phi(\F,\Cdot,\infty)+\Theta_{t-u}\phi(\M,\infty,\ast)-\Theta_{t-u}\phi(\FM,\Cdot,\ast)] \rho^\infty_{\bar S_u}(\Cdot,\ast) \mathbf{1}_\F, \mathscr{Z}_u\Big)_{\Cdot} \mathbf{1}_\M, \bar S_u\Big)_\ast du \\
\notag&\q- \int_0^t \Big(\Big( [\Theta_{t-u}\phi(\F,\Cdot,\infty)+\Theta_{t-u}\phi(\M,\infty,\ast)-\Theta_{t-u}\phi(\FM,\Cdot,\ast)] \rho^\infty_{\bar S_u}(\Cdot,\ast) \mathbf{1}_\F, \bar S_u\Big)_{\Cdot} \mathbf{1}_\M, \mathscr{Z}_u\Big)_\ast du \\
&\q+ \int_0^t (\Theta_{t-u}\phi, d\tilde M^\infty_u).
\label{E:phiZinftyMonogamy}
\end{align}
}
Showing the uniqueness of the solution to \eqref{E:phiZinftyMonogamy} and that \eqref{E:phiZinftyMonogamy} is equivalent to \eqref{E:fZinftyMonogamy} completes the proof. 
\end{proof}

%%%%%%%%%%%%%%%%%%%%%%%%%%%%

\end{document}